\documentclass[11pt]{amsart}%
\usepackage{graphicx}
\usepackage{amscd}
\usepackage{amsmath}
\usepackage{mathtools}
\usepackage{amsfonts}
\usepackage{enumerate}
\usepackage{amssymb}
\usepackage{color}%
\providecommand{\U}[1]{\protect\rule{.1in}{.1in}}
\newtheorem{theorem}{Theorem}
\theoremstyle{plain}

\newtheorem{proposition}{Proposition}[section]
\newtheorem{remark}{Remark}[section]

\numberwithin{equation}{section}
\numberwithin{theorem}{section}

\begin{document}
\title[Spectral Study of the $X_{1}$-Jacobi and Jacobi Differential Expressions]{A Spectral Study of the Second-Order Exceptional $X_{1}$-Jacobi Differential
Expression and a Related Non-classical Jacobi Differential Expression}
\author{Constanze Liaw}
\address{Department of Mathematics, Baylor University, One Bear Place \#97328, Waco, TX 76798-7328}
\email{Constanze\_Liaw@baylor.edu, Lance\_Littlejohn@baylor.edu,
Jessica\_Stewart@baylor.edu, Quinn\_Wicks@baylor.edu}
\author{Lance L.~Littlejohn}
\author{Jessica Stewart}
\author{Quinn Wicks}
\thanks{Constanze Liaw is partially supported by the NSF grant DMS-1261687.}
\date{}

\begin{abstract}
The exceptional $X_{1}$-Jacobi differential expression is a second-order
ordinary differential expression with rational coefficients; it was discovered
by G\'{o}mez-Ullate, Kamran and Milson in 2009. In their work, they showed
that there is a sequence of polynomial eigenfunctions $\left\{\widehat{P}%
_{n}^{(\alpha,\beta)}\right\}_{n=1}^{\infty}$ called the exceptional $X_{1}$-Jacobi
polynomials. There is no exceptional $X_{1}$-Jacobi polynomial of degree zero.
These polynomials form a complete orthogonal set in the weighted Hilbert space
$L^{2}((-1,1);\widehat{w}_{\alpha,\beta}),$ where $\widehat{w}_{\alpha,\beta}$
is a positive rational weight function related to the classical Jacobi weight.
Among other conditions placed on the parameters $\alpha$ and $\beta,$ it is
required that $\alpha,\beta>0.$ In this paper, we develop the spectral theory
of this expression in $L^{2}((-1,1);\widehat{w}_{\alpha,\beta})$. We also
consider the spectral analysis of the `extreme' non-exceptional case, namely
when $\alpha=0$. In this case, the polynomial solutions are the non-classical
Jacobi polynomials $\left\{  P_{n}^{(-2,\beta)}\right\}  _{n=2}^{\infty}.$ We
study the corresponding Jacobi differential expression in several Hilbert
spaces, including their natural $L^{2}$ setting and a certain Sobolev space
$S$ where the full sequence $\left\{  P_{n}^{(-2,\beta)}\right\}
_{n=0}^{\infty}$ is studied and a careful spectral analysis of the Jacobi
expression is carried out.

\end{abstract}
\maketitle

\section{Introduction}

In 2009, G\'{o}mez-Ullate, Kamran, and Milson \cite{KMUG} (see also
\cite{KMUG1, KMUG2, KMUG3, KMUG4, KMUG5, KMUG6}) characterized all polynomial
sequences $\left\{  p_{n}\right\}  _{n=1}^{\infty},$ with $\deg p_{n}=n \ge1$,
which satisfy the following conditions:

\begin{enumerate}
\item[(i)] there exists a second-order differential expression%
\[
\ell\lbrack y](x)=a_{2}(x)y^{\prime\prime}(x)+a_{1}(x)y^{\prime}%
(x)+a_{0}(x)y(x),
\]
and a sequence of complex numbers $\{\lambda_{n}\}_{n=1}^{\infty}$ such that
$y=p_{n}(x)$ is a solution of
\[
\ell\lbrack y](x)=\lambda_{n}y(x)\quad(n\in\mathbb{N});
\]
each coefficient $a_{i}(x),$ $i=0,1,2,$ is a function of the independent
variable $x$ and does not depend on the degree of the polynomial eigenfunctions;

\item[(ii)] if $C$ is any non-zero constant, $y(x)\equiv C$ is \textit{not} a
solution of $\ell\lbrack y](x)=\lambda y(x)$ for any $\lambda\in\mathbb{C};$

\item[(iii)] there exists an open interval $I$ and a positive Lebesgue
measurable function $w(x)$ $(x\in I)$ such that
\[
\int_{I}p_{n}(x)p_{m}w(x)dx=K_{n}\delta_{n,m},
\]
where $K_{n}>0$ for each $n\in\mathbb{N}$ and $\delta_{n,m}$ is the standard
Kronecker delta symbol; that is to say, $\{p_{n}\}_{n=1}^{\infty}$ is
orthogonal with respect to $w$ on the interval $I;$

\item[(iv)] all moments $\{\mu_{n}\}_{n=0}^{\infty}$ of $w,$ defined by
\[
\mu_{n}=\int_{I}x^{n}w(x)dx\quad(n=0,1,2,\ldots),
\]
exist and are finite.
\end{enumerate}

Up to a complex linear change of variable, the authors in \cite{KMUG} show
that the only solutions to this classification problem are the
\textit{exceptional} $X_{1}$-Laguerre and $X_{1}$-Jacobi polynomials. Their
results are spectacular and remarkable; indeed, it was believed, due to the
`Bochner' classification (see \cite{Bochner}, \cite{Lesky} and \cite{Routh}),
that among the class of all orthogonal polynomials, only the Hermite,
Laguerre, and Jacobi polynomials, satisfy second-order differential equations
and are orthogonal with respect to a positive-definite inner product of the
form%
\[
(p,q)=\int_{\mathbb{R}}p(x)\overline{q}(x)W(x)dx.
\]
We remark that two excellent texts dealing with the subject of orthogonal
polynomials are the classical texts of \cite{Chihara} and \cite{Szego}.

Even though the authors in \cite{KMUG} introduce the notion of exceptional
polynomials via Sturm-Liouville theory, the path that they followed to their
discovery was motivated by their interest in quantum mechanics, specifically
with their intent to extend exactly solvable and quasi-exactly solvable
potentials beyond the Lie algebraic setting. It is important to note as well
that the work in \cite{KMUG} was not originally motivated by orthogonal
polynomials although they set out to construct potentials that would be
solvable by polynomials which fall outside the realm of the classical theory
of orthogonal polynomials. To further note, their work was inspired by the
paper of Post and Turbiner \cite{Post-Turbiner} who formulated a
\textit{generalized Bochner problem} of classifying the linear differential
operators in one variable leaving invariant a given vector space of polynomials.

The $X_{1}$-Laguerre and $X_{1}$-Jacobi polynomials, as well as subsequent
generalizations, are exceptional in the sense that they start at degree $\ell$
$(\ell\geq1)$ instead of degree $0$, thus avoiding the restrictions of the
Bochner classification, but still satisfy second-order differential equations
of spectral type. Reformulation within the framework of one-dimensional
quantum mechanics and shape invariant potentials is considered by various
other authors; for example, see \cite{Odake-Sasaki} and \cite{Quesne}.
Furthermore, the two second-order differential equations that they discover in
their $X_{1}$ classification are important examples illustrating the
Stone-von\ Neumann theory \cite[Chapter 12]{DS} and the Glazman-Krein-Naimark
theory (see \cite{Akhiezer-Glazman} and \cite[Section 18]{Naimark}) of
differential operators.

In this paper, we study the exceptional $X_{1}$-Jacobi expression for all
possible parameter choices in various Hilbert spaces. We also consider this
expression, the corresponding orthogonal polynomials and the self-adjoint
theory for the \textit{extreme} choice of parameters $\alpha=0$ or $\beta=0$.
The corresponding operators and their spectral analysis are not captured by
the generalized Bochner classification and we apply a multitude of techniques
to accomplish our goals.

The contents of this paper are as follows. In Section \ref{Section Two}, we
introduce the exceptional $X_{1}$-Jacobi polynomials and differential
expression and briefly review properties of these polynomials. Section
\ref{Section Three} deals with standard properties of the exceptional $X_{1}%
$-Jacobi differential expression $\widehat{\ell}_{\alpha,\beta}[\cdot]$ in its
natural setting $L^{2}((-1,1);\widehat{w}_{\alpha,\beta}),$ where
$\widehat{w}_{\alpha,\beta}$ is the orthogonalizing weight function for the
exceptional $X_{1}$-Jacobi polynomials. This leads to the construction, in
Section \ref{Section Four}, of a certain self-adjoint operator $\widehat{T}%
_{\alpha,\beta},$ generated by $\widehat{\ell}_{\alpha,\beta}[\cdot],$ in
$L^{2}((-1,1);\widehat{w}_{\alpha,\beta})$ (see Theorem \ref{t-abge1}). In
Section \ref{Section Five}, we begin our analysis of the `extreme' case
$\alpha=0.$ This choice gets us closer to the realm of classical orthogonal
polynomials; indeed the weight function in this case simplifies to
$w_{-2,\beta}(x)=(1-x)^{-2}(1+x)^{\beta},$ which is the weight function for
the non-classical Jacobi polynomials $\left\{  P_{n}^{(-2,\beta)}\right\}  .$
Various important facts about the associated Jacobi differential expression,
which we denote by $m_{-2,\beta}[\cdot],$ are discussed in Section
\ref{Section Six}. These properties are used in Section \ref{Section Seven} to
construct the self-adjoint operator $T_{-2,\beta},$ generated by $m_{-2,\beta
}[\cdot],$ having the Jacobi polynomials $\left\{  P_{n}^{(-2,\beta)}\right\}
_{n=2}^{\infty}$ as eigenfunctions; see Theorem
\ref{Self-Adjointness of Special Jacobi Operator}. We remark that it is not
possible for the Jacobi polynomials $P_{n}^{(-2,\beta)}$ of degrees $0$ and
$1$ to belong to $L^{2}((-1,1);w_{-2,\beta}).$ Also, in Section
\ref{Section Seven}, we show (Theorem \ref{positivity theorem}) that
$T_{-2,\beta}$ is bounded below by the identity operator $I$ in $L^{2}%
((-1,1);w_{-2,\beta}).$ This result will be critical for our analysis in the
last two sections of the paper. Section \ref{Section Eight} gives a short
description of abstract left-definite theory, a subject that is instrumental
in the last two sections. Kwon and Littlejohn \cite{Kwon-Littlejohn}
discovered a Sobolev inner product in which the \textit{entire }Jacobi
sequence $\left\{  P_{n}^{(-2,\beta)}\right\}  _{n=0}^{\infty}$ is orthogonal
but, for reasons that will be made clearer later, we must require $\beta
\neq0.$ This inner product and properties of the corresponding Sobolev space
$S$ are discussed in Section \ref{Section Nine}. Lastly, in Section
\ref{Section Ten}, we construct (Theorem \ref{Self-Adjoint Operator in S}) a
self-adjoint operator $T,$ generated by the differential expression
$m_{-2,\beta}[\cdot]$, having the Jacobi polynomials $\left\{  P_{n}%
^{(-2,\beta)}\right\}  _{n=0}^{\infty}$ as eigenfunctions. This construction,
essentially, uses all of the results proven in the previous sections.

\section{The Exceptional $X_{1}$-Jacobi Polynomials\label{Section Two}}

The exceptional $X_{1}$-Jacobi differential expression is defined to be%
\begin{equation}
\widehat{\ell}_{\alpha,\beta}[y](x):=(x^{2}-1)y^{\prime\prime}(x)+2a\left(
\frac{1-bx}{b-x}\right)  \left(  (x-c)y^{\prime}(x)-y(x)\right)  \quad
(x\in(-1,1)), \label{X_1JacobiExpression}%
\end{equation}
where%
\begin{equation}
\alpha,\beta\in(-1,\infty),\quad\alpha\neq\beta,\quad\text{and }%
\quad\text{\textrm{sgn}}(\alpha)=\text{\textrm{sgn}}(\beta),
\label{Parameter Conditions1}%
\end{equation}
and%
\begin{equation}
a:=\frac{1}{2}(\beta-\alpha)\,,\quad b:=\frac{\beta+\alpha}{\beta-\alpha
}\,,\quad c:=b+\frac{1}{a}=\frac{\beta+\alpha+2}{\beta-\alpha}\,.
\label{Parameter Conditions2}%
\end{equation}
Notice that, from (\ref{Parameter Conditions1}), that it is not possible for
$\alpha=0$ or $\beta=0.$ Later, in Section \ref{Section Five} and onwards, we
do allow for $\alpha=0$ or $\beta=0.$

Observe that the conditions in (\ref{Parameter Conditions1}) imply that
$\left\vert b\right\vert >1.$ Indeed suppose, to the contrary, that
$\left\vert b\right\vert \leq1;$ that is to say,
\[
-1\leq\frac{\beta+\alpha}{\beta-\alpha}\leq1.
\]
If $\alpha>\beta,$ we see that the above inequality yields $-\beta+\alpha
\geq\beta+\alpha\geq\beta-\alpha,$ which in turn implies $\beta\leq0$ and
$\alpha\geq0.$ Since the case $\alpha=0$ or $\beta=0$ is not possible, we see
that \textrm{sgn}$(\beta)=-\mathrm{sgn}(\alpha),$ contradicting
(\ref{Parameter Conditions1}). The case $\alpha<\beta$ can be dealt with similarly.

The exceptional $X_{1}$-Jacobi polynomials $\left\{  \widehat{P}_{n}%
^{(\alpha,\beta)}\right\}  _{n=1}^{\infty}$ are eigenfunctions of
$\widehat{\ell}_{\alpha,\beta}[\cdot];$ specifically%
\[
\widehat{\ell}_{\alpha,\beta}[\widehat{P}_{n}^{(\alpha,\beta)}%
](x)=(n-1)(\alpha+\beta+n)\widehat{P}_{n}^{(\alpha,\beta)}(x)\quad
(n\in\mathbb{N}).
\]
Moreover, they show that $\left\{  \widehat{P}_{n}^{(\alpha,\beta)}\right\}
_{n=1}^{\infty}$ forms a complete orthogonal set in the Hilbert space
\[
L^{2}((-1,1);\widehat{w}_{\alpha,\beta}):=\left\{  f:(-1,1)\rightarrow
\mathbb{C}\mid f\text{ is Lebesgue measurable and }\left\Vert f\right\Vert
_{\widehat{w}_{\alpha,\beta}}<\infty\right\}  ,
\]
with norm and inner product defined, respectively, by
\begin{equation}
\left\Vert f\right\Vert _{\widehat{w}_{\alpha,\beta}}:=\left(  \int_{-1}%
^{1}\left\vert f(x)\right\vert ^{2}\widehat{w}_{\alpha,\beta}(x)\,dx\right)
^{1/2}\quad(f\in L^{2}((-1,1);\widehat{w}_{\alpha,\beta})) \label{X_1Norm}%
\end{equation}
and%
\begin{equation}
(f,g)_{\widehat{w}_{\alpha,\beta}}:=\int_{-1}^{1}f(x)\overline{g}%
(x)\widehat{w}_{\alpha,\beta}(x)\,dx\quad(f,g\in L^{2}((-1,1);\widehat{w}%
_{\alpha,\beta})), \label{X_1IP}%
\end{equation}
where%

\begin{equation}
\widehat{w}_{\alpha,\beta}(x):=\frac{(1-x)^{\alpha}(1+x)^{\beta}}{(x-b)^{2}%
}\,\quad(x\in(-1,1)). \label{weight}%
\end{equation}
Since $|b|$ $>1,$ the term $(x-b)^{-2}$ in the weight function $\widehat{w}%
_{\alpha,\beta}$ is bounded on $\left[  -1,1\right]  $; consequently the
moments of $\widehat{w}_{\alpha,\beta}$ all exist and are finite for all
$\alpha$ and $\beta$ satisfying the conditions in (\ref{Parameter Conditions1}).

\begin{remark}
The term $x-b$ that appears in the denominator of both $($%
\ref{X_1JacobiExpression}$)$ and $($\ref{weight}$)$ is a multiple of the
degree one Jacobi polynomial $P_{1}^{(-\alpha-1,\beta-1)}(x);$ in fact
\[
x-b=\frac{2}{\beta-\alpha}P_{1}^{(-\alpha-1,\beta-1)}(x).
\]
In \cite{KMUG6} and \cite{Odake-Sasaki}, the authors study more general
exceptional $X_{m}$-Jacobi polynomials; these polynomials are orthogonal with
respect to the weight function%
\[
\widehat{w}_{\alpha,\beta,m}(x)=\frac{(x-1)^{\alpha}(1+x)^{\beta}}%
{(P_{m}^{(-\alpha-1,\beta-1)}(x))^{2}}.
\]
Notice that, when $m=1,$ this weight reduces, essentially, to $($%
\ref{weight}$)$.
\end{remark}

These exceptional $X_{1}$-Jacobi polynomials are explicitly given by
\begin{equation}
\widehat{P}_{n}^{(\alpha,\beta)}(x)=-\frac{1}{2}(x-b)P_{n-1}^{(\alpha,\beta
)}(x)+\frac{bP_{n-1}^{(\alpha,\beta)}(x)-P_{n-2}^{(\alpha,\beta)}(x)}%
{\alpha+\beta+2n-2}\qquad(n\in\mathbb{N};P_{-1}^{(\alpha,\beta)}(x)=0),
\label{recursion}%
\end{equation}
where $\left\{  P_{n}^{(\alpha,\beta)}\right\}  _{n=1}^{\infty}$ are the
classical Jacobi polynomials, defined by
\begin{equation}
P_{n}^{(\alpha,\beta)}(x)=2^{-n}\sum_{k=0}^{n}{\binom{n+\alpha}{n-k}}%
{\binom{n+\beta}{k}}(x-1)^{k}(x+1)^{n-k}. \label{Jacobi polynomials}%
\end{equation}
For the sake of completeness, we list a few of these exceptional $X_{1}%
$-Jacobi polynomials:%
\[%
\begin{array}
[c]{ll}%
\widehat{P}_{1}^{(\alpha,\beta)}(x)= & -\dfrac{1}{2}x-\dfrac{\alpha+\beta
+2}{2(\alpha-\beta)}\medskip\\
\widehat{P}_{2}^{(\alpha,\beta)}(x)= & -\dfrac{\alpha+\beta+2}{4}x^{2}%
-\dfrac{\alpha(\alpha+2)+\beta(\beta+2)}{2(\alpha-\beta)}x-\dfrac{\alpha
+\beta+2}{4}\medskip\\
\widehat{P}_{3}^{(\alpha,\beta)}(x)= & -\dfrac{(\alpha+\beta+3)(\alpha
+\beta+4)}{16}x^{3}-\dfrac{(\alpha+\beta+3)(3\alpha^{2}+6\alpha-2\alpha
\beta+3\beta^{2}+6\beta)}{16(\alpha-\beta)}x^{2}\medskip\\
& -\dfrac{(3\alpha^{2}+9\alpha+2\alpha\beta+3\beta^{2}+9\beta)}{16}x\medskip\\
& -\dfrac{(\alpha^{3}+\alpha^{2}-6\alpha-\alpha^{2}\beta-6\alpha\beta
-6\beta-\alpha\beta^{2}+\beta^{3}+\beta^{2})}{16(\alpha-\beta)}.
\end{array}
\]
The norms of these polynomials are explicitly given by
\[
\left\Vert \widehat{P}_{n}^{(\alpha,\beta)}\right\Vert _{\widehat{w}%
_{\alpha,\beta}}^{2}=\left(  \frac{2^{\alpha+\beta+1}(\alpha+n)(\beta
+n)}{4(\alpha+n+1)(\beta+n-1)(\alpha+\beta+2n-1)}\right)  \left(  \frac
{\Gamma(\alpha+n)\Gamma(\beta+n)}{\Gamma(n)\Gamma(\alpha+\beta+n)}\right)
\quad(n\in\mathbb{N}).
\]

In \cite{Gomez-Ullate-Marcellan-Milson}, the authors establish the location
and asymptotic behavior of the roots of the exceptional $X_{1}$-Jacobi
polynomials. Indeed, they show that there are $n-1$ simple roots of
$\widehat{P}_{n}^{(\alpha,\beta)}(x)$ $(n\in\mathbb{N}_{0})$ lying in the
interval $(-1,1)$ and there is exactly one negative root. Asymptotically, as
$n\rightarrow\infty,$ the $n-1$ roots of $\widehat{P}_{n}^{(\alpha,\beta)}(x)$
in $(-1,1)$ converge to the roots of the classical Jacobi polynomial
$P_{n-1}^{(\alpha,\beta)}(x)$ while the negative root of $\widehat{P}%
_{n}^{(\alpha,\beta)}(x)$ converges to the root of $P_{1}^{(-\alpha
-1,\beta-1)}(x);$ for further details, see \cite[Proposition 5.3, Proposition
5.4 and Corollary 5.1, page 493]{Gomez-Ullate-Marcellan-Milson}.

\section{Properties of the Exceptional $X_{1}$-Jacobi Differential
Expression\label{Section Three}}

Some properties of this expression were developed by Everitt in \cite{Everitt}%
; we reproduce some of his results in this section. Both endpoints $x=\pm1$
are regular singular endpoints, in the sense of Frobenius, of the exceptional
$X_{1}$-Jacobi differential expression $\widehat{\ell}_{\alpha,\beta}[\cdot].$
The Frobenius indicial equation at $x=1$ is $r(r+\alpha)=0.$ Therefore, two
linearly independent solutions of $\widehat{\ell}_{\alpha,\beta}[y]=0$ behave
asymptotically near $x=1$ like
\[
z_{1}(x)=1\text{ and }z_{2}(x)=(x-1)^{-\alpha}\text{ .}%
\]
For all feasible values of $\alpha$ and $\beta,$ we have
\[
\int_{0}^{1}|z_{1}(x)|^{2}\widehat{w}_{\alpha,\beta}(x)dx<\infty.
\]
However,
\[
\int_{0}^{1}|z_{2}(x)|^{2}\widehat{w}_{\alpha,\beta}(x)dx<\infty,
\]
only when $-1<\alpha<1$. Consequently, at $x=1,$ the expression $\widehat{\ell
}_{\alpha,\beta}[\cdot]$ is limit-point for $\alpha\geq1$ and limit-circle
when $-1<\alpha<1.$ The analysis at $x=-1$ is similar, in this case,
$\widehat{\ell}_{\alpha,\beta}[\cdot]$ is limit-point for $\beta\geq1$ and
limit-circle in the case $-1<\beta<1.$

\bigskip In Lagrangian symmetric form, the $X_{1}$-Jacobi differential
expression (\ref{X_1JacobiExpression}) is given by
\begin{align}
\widehat{\ell}_{\alpha,\beta}[y](x) =\frac{1}{\widehat{w}_{\alpha,\beta}(x)}
&  \left(  -\left(  \frac{(1-x)^{\alpha+1}(1+x)^{\beta+1}}{(x-b)^{2}}%
y^{\prime}(x)\right)  ^{\prime}\right. \label{symmetric form}\\
&  +\left.  \frac{2a(x-c)(bx-1)(1-x)^{\alpha}(1+x)^{\beta}}{(x-b)^{3}%
}y(x)\right)  ,\nonumber
\end{align}

The maximal domain associated with $\widehat{\ell}_{\alpha,\beta}[\cdot]$ in
the Hilbert space $L^{2}((-1,1);\widehat{w}_{\alpha,\beta})$ is
\begin{equation}
\widehat{\Delta}:=\left\{  f:(0,\infty)\rightarrow\mathbb{C}\mid f,f^{\prime
}\in AC_{\text{\textrm{loc}}}(-1,1);f,\widehat{\ell}_{\alpha,\beta}[f]\in
L^{2}((-1,1);\widehat{w}_{\alpha,\beta})\right\}  . \label{MaxDomain}%
\end{equation}
The associated maximal operator
\[
\widehat{T}_{\max}:\mathcal{D}(\widehat{T}_{\max})\subset L^{2}%
((-1,1);\widehat{w}_{\alpha,\beta})\rightarrow L^{2}((-1,1);\widehat{w}%
_{\alpha,\beta}),
\]
is defined by
\begin{align}
\widehat{T}_{\max}f  &  =\widehat{\ell}_{\alpha,\beta}%
[f]\label{Maximal Operator}\\
f\in\mathcal{D}(\widehat{T}_{\max}) :  &  =\widehat{\Delta}.\nonumber
\end{align}
For $f,g\in\widehat{\Delta},$ Green's formula can be written as%
\begin{equation}
\int_{-1}^{1}\widehat{\ell}_{\alpha,\beta}[f](x)\overline{g}(x)\widehat{w}%
_{\alpha,\beta}(x)dx=[f,g]_{\widehat{w}_{\alpha,\beta}}(x)\mid_{x=-1}%
^{x=1}+\int_{-1}^{1}f(x)\widehat{\ell}_{\alpha,\beta}[\overline{g}%
](x)\widehat{w}_{\alpha,\beta}(x)dx, \label{greens}%
\end{equation}
where $[\cdot,\cdot]$ is the sesquilinear form defined by%
\begin{equation}
\lbrack f,g]_{\widehat{w}_{\alpha,\beta}}(x):=\frac{(1-x)^{\alpha
+1}(1+x)^{\beta+1}}{(x-b)^{2}}\left(  f(x)\overline{g}^{\prime}(x)-f^{\prime
}(x)\overline{g}(x)\right)  \quad(-1<x<1), \label{Wronskian}%
\end{equation}
and
\[
\lbrack f,g]_{\widehat{w}_{\alpha,\beta}}(x)\mid_{x=-1}^{x=1}%
:=[f,g]_{\widehat{w}_{\alpha,\beta}}(1)-[f,g]_{\widehat{w}_{\alpha,\beta}%
}(-1).
\]
By definition of $\widehat{\Delta}$, and the classical H\"{o}lder's
inequality, notice that the limits
\[
\lbrack f,g]_{\widehat{w}_{\alpha,\beta}}(-1):=\lim_{x\rightarrow-1^{+}%
}[f,g]_{\widehat{w}_{\alpha,\beta}}(x)\quad\text{ and }\quad\lbrack
f,g]_{\widehat{w}_{\alpha,\beta}}(1):=\lim_{x\rightarrow1^{-}}%
[f,g]_{\widehat{w}_{\alpha,\beta}}(x)
\]
both exist and are finite for each $f,g\in\widehat{\Delta}.$

By standard classical arguments, the maximal domain $\widehat{\Delta}$ is
dense in $L^{2}((-1,1);\widehat{w}_{\alpha,\beta})$; consequently, the adjoint
of $\widehat{T}_{\max}$ exists as a densely defined operator in $L^{2}%
((-1,1);\widehat{w}_{\alpha,\beta}).$ For obvious reasons, the adjoint of
$\widehat{T}_{\max}$ is called the minimal operator associated with
$\widehat{\ell}_{\alpha,\beta}[\cdot]$ and is denoted by $\widehat{T}_{\min}.$
From \cite{Akhiezer-Glazman} or \cite{Naimark}, this minimal operator
$\widehat{T}_{\min}:\mathcal{D}(T_{\min})\subset L^{2}((-1,1);\widehat{w}%
_{\alpha,\beta})\rightarrow L^{2}((-1,1);\widehat{w}_{\alpha,\beta})$ is
defined by%
\begin{align}
\widehat{T}_{\min}f  &  =\widehat{\ell}_{\alpha,\beta}%
[f]\label{Minimal Operator}\\
f\in\mathcal{D}(\widehat{T}_{\min}) :  &  =\{f\in\widehat{\Delta}\mid\lbrack
f,g]_{\widehat{w}_{\alpha,\beta}}\mid_{x=-1}^{x=1}=0\text{ for all }%
g\in\widehat{\Delta}\}.\nonumber
\end{align}
The minimal operator $\widehat{T}_{\min}$ is a closed, symmetric operator in
$L^{2}((-1,1);\widehat{w}_{\alpha,\beta});$ furthermore, because the
coefficients of $\widehat{\ell}_{\alpha,\beta}[\cdot]$ are real,
$\widehat{T}_{\min}$ necessarily has equal deficiency indices $m,$ where $m$
is an integer satisfying $0\leq m\leq2.$ Therefore, from the general Stone-von
Neumann \cite{DS} theory of self-adjoint extensions of symmetric operators,
$\widehat{T}_{\min}$ has self-adjoint extensions. We seek to find the
self-adjoint extension $\widehat{T}$ in $L^{2}((-1,1);\widehat{w}%
_{\alpha,\beta}),$ generated by $\widehat{\ell}_{\alpha,\beta}[\cdot],$ which
has the $X_{1}$-Jacobi polynomials $\left\{\widehat{P}_{n}^{(\alpha,\beta)}%
\right\}_{n=1}^{\infty}$ as eigenfunctions. From the Frobenius analysis discussed at
the beginning of this section, the following Proposition follows immediately.

\begin{proposition}
\label{t-index} Consider the minimal operator $\widehat{T}_{\min}$ in
$L^{2}((-1,1);\widehat{w}_{\alpha,\beta}),$ as defined in $($%
\ref{Minimal Operator}$)$, generated by the exceptional $X_{1}$-Jacobi
differential expression $\widehat{\ell}_{\alpha,\beta}[\cdot]$.

\begin{itemize}
\item[(a)] For $\alpha,\beta\geq1$, the minimal operator $\widehat{T}_{\min}$
has deficiency index $(0,0).$

\item[(b)] For $\alpha\geq1$, and $\beta<1,$ the minimal operator
$\widehat{T}_{\min}$ has deficiency index $(1,1).$ The same is true for
$\alpha<1$ and $\beta\geq1$.

\item[(c)] For $\alpha,\beta<1$, the minimal operator $\widehat{T}_{\min}$ has
deficiency index $(2,2).$
\end{itemize}
\end{proposition}

\section{A Certain Exceptional $X_{1}$-Jacobi Self-Adjoint
Operator\label{Section Four}}

Proposition \ref{t-index} puts us in a position to define the self-adjoint
operator $\widehat{T}_{\alpha,\beta}$ in $L^{2}((-1,1);\widehat{w}%
_{\alpha,\beta})$ having the exceptional $X_{1}$-Jacobi polynomials $\left\{
\widehat{P}_{n}^{(\alpha,\beta)}\right\}  _{n=1}^{\infty}$ as eigenfunctions;
this operator is found by a direct application of the so-called
Glazman-Krein-Naimark theory (see \cite{Akhiezer-Glazman} and \cite{Naimark}).
The one boundary function, when needed, that we choose to generate the
appropriate boundary condition is $g(x)=1.$ When we substitute this function
into the sesquilinear form (\ref{Wronskian}) associated with $\widehat{\ell
}_{\alpha,\beta}[\cdot]$, we see that%
\[
\lbrack f,1]_{\widehat{w}_{\alpha,\beta}}(x)=-\frac{(1-x)^{\alpha
+1}(1+x)^{\beta+1}}{(x-b)^{2}}f^{\prime}(x);
\]
moreover, notice that the boundary condition $\lim_{x\rightarrow1^{-}%
}[f,1]_{\widehat{w}_{\alpha,\beta}}(x)=0$ simplifies to
\[
\lim_{x\rightarrow1^{-}}(1-x)^{\alpha+1}f^{\prime}(x)=0.
\]
An analogous argument works for $x \rightarrow-1^{+}$. We are now ready to
state the following theorem.

\begin{theorem}
\label{t-abge1} The self-adjoint operator $\widehat{T}_{\alpha,\beta}$ in
$L^{2}((-1,1);\widehat{w}_{\alpha,\beta}),$ generated by the exceptional
$X_{1}$-Jacobi differential expression $\widehat{\ell}_{\alpha,\beta}[\cdot]$,
having the exceptional $X_{1}$-Jacobi polynomials $\left\{  \widehat{P}%
_{n}^{(\alpha,\beta)}\right\}  _{n=1}^{\infty}$ as eigenfunctions is
explicitly given by%
\begin{align*}
\widehat{T}_{\alpha,\beta}f  &  =\widehat{\ell}_{\alpha,\beta}[f]\\
f  &  \in\mathcal{D}(\widehat{T}_{\alpha,\beta}),
\end{align*}
where
\begin{equation}
\mathcal{D}(\widehat{T}_{\alpha,\beta})=\left\{
\begin{array}
[c]{ll}%
\widehat{\Delta} & \text{if }\alpha\geq1\text{ and }\beta\geq1\medskip\\
\{f\in\widehat{\Delta}\mid\lim_{x\rightarrow-1^{+}}(1+x)^{\beta+1}f^{\prime
}(x)=0\} & \text{if }\alpha\geq1\text{ and }0<\beta<1\medskip\\
\{f\in\widehat{\Delta}\mid\lim_{x\rightarrow1^{-}}(1-x)^{\alpha+1}f^{\prime
}(x)=0\} & \text{if }0<\alpha<1\text{ and }\beta\geq1\medskip\\%
\begin{array}
[c]{l}%
\!\!\!\{f\in\widehat{\Delta}\mid\lim_{x\rightarrow1^{-}}(1-x)^{\alpha
+1}f^{\prime}(x)\\
\qquad\qquad=\lim_{x\rightarrow-1^{+}}(1+x)^{\beta+1}f^{\prime}(x)=0\}
\end{array}
&
\begin{array}
[c]{l}%
\!\!\text{if }0<\alpha<1\text{ and }0<\beta<1\text{ or}\\
\!\!\text{if }-1<\alpha<0\text{ and }-1<\beta<0.
\end{array}
\end{array}
\right.  \label{BC for Exceptional}%
\end{equation}
The exceptional $X_{1}$-Jacobi polynomials $\left\{  \widehat{P}_{n}%
^{(\alpha,\beta)}\right\}  _{n=1}^{\infty}$ form a complete set of
eigenfunctions of $\widehat{T}_{\alpha,\beta}$ in $L^{2}((-1,1);\widehat{w}%
_{\alpha,\beta})$. Furthermore the spectrum $\sigma(\widehat{T}_{\alpha,\beta
})$ of $\widehat{T}_{\alpha,\beta}$ is pure discrete spectrum consisting of
the simple eigenvalues%
\[
\sigma(\widehat{T}_{\alpha,\beta})=\sigma_{p}(\widehat{T}_{\alpha,\beta
})=\{(n-1)(\alpha+\beta+n)\mid n\in\mathbb{N}\}.
\]

\end{theorem}

\section{The `Extreme' Case $\alpha=0$ and $\beta>-1:$ Non-classical Jacobi
Polynomials\label{Section Five}}

We now study the situation when $\alpha=0$ and $\beta>-1$ in the exceptional
$X_{1}$-Jacobi case; the reader will recall that this situation was not
allowed in our earlier analysis from the conditions given in
(\ref{Parameter Conditions1}). There is the analogous case $\beta=0$ and
$\alpha>-1$ which we will not address in this paper. We remark that there do
not appear to be any interesting extreme cases for exceptional $X_{m}$-Jacobi
or $X_{m}$-Laguerre polynomials when $m>1.$ There is an interesting extreme
case for the exceptional $X_{1}$-Laguerre polynomials. This was reported on,
albeit in incomplete details, in \cite{X1-Laguerre}.

When $\alpha=0$ and $\beta>-1,$ we see from (\ref{Parameter Conditions1}) that%
\[
a=\beta/2,\text{ }b=1,\text{ and }c=(\beta+2)/\beta.
\]
With these choices, we note that the differential expression
(\ref{X_1JacobiExpression}) becomes%
\begin{equation}
\widehat{\ell}_{0,\beta}[y](x)=(x^{2}-1)y^{\prime\prime}(x)+(\beta
x-\beta-2)y^{\prime}(x)-\beta y(x)\quad(x\in(-1,1)).
\label{Unperturbed Jacobi DE}%
\end{equation}
For reasons that will be made clearer later, we perturb the coefficient of $y$
(by adding $(1+\beta)y(x))$ and we will instead study the Jacobi expression%
\begin{equation}
m_{-2,\beta}[y](x):=(x^{2}-1)y^{\prime\prime}(x)+(\beta x-\beta-2)y^{\prime
}(x)+y(x)\quad(x\in(-1,1)). \label{JacobiDE-special}%
\end{equation}
Indeed, adding this term will affect only the spectrum but not the
eigenfunctions. The weight function (\ref{weight}) in this case becomes%
\begin{equation}
w_{-2,\beta}(x):=(1-x)^{-2}(1+x)^{\beta}\quad(x\in(-1,1)).
\label{Jacobiweight-special}%
\end{equation}
This differential expression and weight are precisely the Jacobi differential
expression and Jacobi weight for the \textit{non-classical} Jacobi case
$(\alpha,\beta)=(-2,\beta).$

Even though this is a non-classical Jacobi case, the differential equation%
\[
m_{-2,\beta}[y](x)=\lambda_{n}y
\]
does have a polynomial solution $y=P_{n}^{(-2,\beta)}(x)$ of degree $n$ for
each $n\in\mathbb{N}_{0}.$ If fact,%
\begin{equation}
P_{n}^{(-2,\beta)}(x)=\left\{
\begin{array}
[c]{ll}%
1 & \text{if }n=0\medskip\\
\beta x-\beta-2 & \text{if }n=1\medskip\\
\dfrac{(n+\beta)(n+\beta-1)}{4n(n-1)}(1-x)^{2}P_{n-2}^{(2,\beta)}(x) &
\text{if }n\geq2,
\end{array}
\right.  \label{Jacobi poly - special}%
\end{equation}
where $\left\{P_{n}^{(2,\beta)}\right\}_{n=0}^{\infty}$ are the classical Jacobi
polynomials defined in (\ref{Jacobi polynomials}). Moreover
\[
m_{-2,\beta}[P_{n}^{(-2,\beta)}](x)=\lambda_{n}P_{n}^{(-2,\beta)}(x),
\]
where%
\begin{equation}
\lambda_{n}=n^{2}+(\beta-1)n+1\quad(n\in\mathbb{N}_{0}).
\label{Eigenvalues-special case}%
\end{equation}

\begin{remark}
Letting $\alpha=0$ in the explicit representation $($\ref{recursion}$)$ of
$\widehat{P}_{n}^{(\alpha,\beta)}(x)$, we find that
\[
\widehat{P}_{n}^{(0,\beta)}(x)=-\frac{1}{2}(x-1)P_{n-1}^{(0,\beta)}%
(x)+\frac{P_{n-1}^{(0,\beta)}(x)-P_{n-2}^{(0,\beta)}(x)}{\beta+2n-2}.
\]
We omit the details but it can be shown that, for $n\geq1,$ $\widehat{P}%
_{n}^{(0,\beta)}(x)$ is a multiple of the non-classical Jacobi polynomial
$P_{n}^{(-2,\beta)}(x),$ defined in $($\ref{Jacobi poly - special}$).$
\end{remark}

\begin{remark}
In $($\ref{Jacobi poly - special}$),$ the non-classical Jacobi polynomials
$P_{n}^{(-2,\beta)},$ for $n\geq2,$ are expressed in terms of the classical
Jacobi polynomials $P_{n-2}^{(2,\beta)};$ this is a well-known connection
$($see \cite[Chapter 4, (4.22.2)]{Szego}$)$. These Jacobi polynomials
$\left\{P_{n}^{(-2,\beta)}\right\}_{n=2}^{\infty}$ satisfy the orthogonality relationship%
\[
\int_{-1}^{1}P_{n}^{(-2,\beta)}(x)P_{m}^{(-2,\beta)}(x)w_{-2,\beta}%
(x)dx=\frac{2^{\beta-1}\Gamma(n-1)\Gamma(n+\beta+1)}{n!(2n+\beta
-1)\Gamma(n+\beta-1)}\delta_{n,m}\quad(n,m\geq2).
\]

\end{remark}

\begin{remark}
\label{beta not zero}Beginning in Section \ref{Section Nine}, we will require
that the set $\left\{P_{n}^{(-2,\beta)}\right\}_{n=0}^{\infty}$ is algebraically
complete; that is, $\deg(P_{n}^{(-2,\beta)})=n$ for each $n\in\mathbb{N}_{0}$
so $\left\{P_{n}^{(-2,\beta)}\right\}_{n=0}^{\infty}$ is a basis for the space
$\mathcal{P}$ of all real-valued polynomials. From $($%
\ref{Jacobi poly - special}$),$ in order for $\deg(P_{1}^{(-2,\beta)})=1,$ we
need $\beta\neq0.$ Thus, starting in Section \ref{Section Nine}, we will
additionally assume $\beta\neq0.$
\end{remark}

Let $L^{2}((-1,1);w_{-2,\beta})$ be the Hilbert space defined by
\[
L^{2}((-1,1);w_{-2,\beta})=\{f:(-1,1)\rightarrow\mathbb{C}\mid f\text{ is
Lebesgue measurable and }\left\Vert f\right\Vert _{w_{-2,\beta}}<\infty\},
\]
where the norm is%
\[
\left\Vert f\right\Vert _{w_{-2,\beta}}=\left(  \int_{-1}^{1}\left\vert
f(x)\right\vert ^{2}w_{-2,\beta}(x)dx\right)  ^{1/2}\quad(f\in L^{2}%
((-1,1);w_{-2,\beta}))
\]
and inner product is%
\[
(f,g)_{w_{-2,\beta}}=\int_{-1}^{1}f(x)\overline{g}(x)w_{-2,\beta}%
(x)dx\quad(f,g\in L^{2}((-1,1);w_{-2,\beta})).
\]

\begin{theorem}
\label{Density of Jacobi - Special}The polynomials $P_{j}^{(-2,\beta)}$
$\notin L^{2}((-1,1);w_{-2,\beta})$ for $j=0,1.$ However, $\left\{P_{n}%
^{(-2,\beta)}\right\}_{n=2}^{\infty}\subset L^{2}((-1,1);w_{-2,\beta});$ moreover,
$\mathrm{span}\left\{P_{n}^{(-2,\beta)}\right\}_{n=2}^{\infty}$ is a complete orthogonal
set in $L^{2}((-1,1);w_{-2,\beta}).$ The last statement is equivalent to
saying
\[
\mathrm{span}\left\{  p\in\mathcal{P}\mid p\text{ is a polynomial of }\deg
\geq2\text{ with }p(1)=p^{\prime}(1)=0\right\}
\]
is dense in $L^{2}((-1,1);w_{-2,\beta})$.

\begin{proof}
The singular term $(1-x)^{-2}$ in the weight function $w_{-2,\beta}(x)$
prevents $P_{j}^{(-2,\beta)}$ $($when $\beta\neq0)$ from belonging to
$L^{2}((-1,1);w_{-2,\beta})$ when $j=0,1.$ The equivalence of the two
statements in this theorem is immediate from $($\ref{Jacobi poly - special}%
$)$; we will prove the second statement. Let $\varepsilon>0$ and $f\in
L^{2}((-1,1);w_{-2,\beta})$. Note that
\[
\int_{-1}^{1}\left\vert f(x)\right\vert ^{2}(1-x)^{-2}(1+x)^{\beta}%
dx=\int_{-1}^{1}\left\vert \frac{f(x)}{(1-x)^{2}}\right\vert ^{2}%
(1-x)^{2}(1-x)^{\beta}dx;
\]
by letting $w_{2,\beta}(x)=(1-x)^{2}(1+x)^{\beta},$ we see that
\[
f\in L^{2}((-1,1);w_{-2,\beta})\text{ if and only if }\frac{f}{(1-x)^{2}}\in
L^{2}((-1,1);w_{2,\beta}).
\]
Since polynomials are dense in $L^{2}((-1,1);w_{2,\beta}),$ there exists
$p\in\mathcal{P}$ such that
\[
\varepsilon^{2}>\int_{-1}^{1}\left\vert \frac{f(x)}{(1-x)^{2}}-p(x)\right\vert
^{2}w_{2,\beta}(x)\,dx.
\]
Define $q(x)=p(x)(1-x)^{2}$ so $q\in\mathcal{P}$ and $q(1)=q^{\prime}(1)=0$.
Moreover,
\[
\left\Vert f-q\right\Vert _{w_{-2,\beta}}^{2}=\int_{-1}^{1}\left\vert
\frac{f(x)}{(1-x)^{2}}-p(x)\right\vert ^{2}w_{2,\beta}(x)\,dx<\varepsilon
^{2},
\]
proving the desired result.
\end{proof}
\end{theorem}

At this point, we remark that Littlejohn and Kwon \cite{Kwon-Littlejohn}
showed that the \textit{entire} sequence of non-classical Jacobi polynomials
$\left\{P_{n}^{(-2,\beta)}\right\}_{n=0}^{\infty}$ are orthogonal with respect to the
Sobolev inner product%
\begin{align}
\phi(f,g)  &  :=f(1)\overline{g}(1)+\frac{2}{\beta}(f^{\prime}(1)\overline
{g}(1)+f(1)\overline{g}^{\prime}(1))\label{Sobolev IP}\\
&  +\left(  1+\frac{4}{\beta^{2}}\right)  f^{\prime}(1)\overline{g}^{\prime
}(1)+\int_{-1}^{1}f^{\prime\prime}(x)\overline{g}^{\prime\prime}%
(x)(1+x)^{\beta+2}\,dx\,.\nonumber
\end{align}
Since
\[
\phi(f,g)=\left(  f(1)+\frac{2}{\beta}f^{\prime}(1)\right)  \left(
\overline{g}(1)+\frac{2}{\beta}\overline{g}^{\prime}(1)\right)  +f^{\prime
}(1)\overline{g}^{\prime}(1)+\int_{-1}^{1}p^{\prime\prime}(x)\overline
{g}^{\prime\prime}(x)(1+x)^{\beta+2}dx,
\]
it is clear that $\phi(\cdot,\cdot)$ is an inner product. Also, it is a
straightforward exercise to show that $\left\{P_{n}^{(-2,\beta)}\right\}_{n=0}^{\infty}$
is orthogonal with respect to $\phi(\cdot,\cdot).$ Later in this paper, we do
a further study of these Jacobi polynomials under this inner product. In
particular, we will identify the appropriate Sobolev space $S$ in which
$\left\{P_{n}^{(-2,\beta)}\right\}_{n=0}^{\infty}$ is a complete orthogonal set.
Moreover, we also construct a self-adjoint operator $T$, generated by
$m_{-2,\beta}[\cdot],$ in $S$ that has $\left\{P_{n}^{(-2,\beta)}\right\}_{n=0}^{\infty}$
as eigenfunctions. This operator $T$ is partially constructed from the
self-adjoint operator $T_{-2,\beta},$ which we now discuss, in $L^{2}%
((-1,1);w_{-2,\beta})$ having the Jacobi polynomials $\left\{P_{n}^{(-2,\beta
)}\right\}_{n=2}^{\infty}$ as eigenfunctions.

\section{Properties of the Non-classical Jacobi Differential Expression
$m_{-2,\beta}[\cdot]$\label{Section Six}}

We now focus our attention to the study of $m_{-2,\beta}[\cdot]$, defined in
(\ref{JacobiDE-special}), in the Hilbert space $L^{2}((-1,1);w_{-2,\beta})$
which is the natural `right-definite' setting for an analytic study.

The Lagrangian symmetric form of $m_{-2,\beta}[\cdot]$ is given by
\begin{equation}
m_{-2,\beta}[y](x)=\frac{1}{w_{-2,\beta}(x)}\left(  -((1-x)^{-1}%
(1+x)^{\beta+1}y^{\prime}(x))^{\prime}+w_{-2,\beta}(x)y(x)\right)  .
\label{Lagrangian form - special}%
\end{equation}
In this case, the maximal domain of $m_{-2,\beta}[\cdot]$ in $L^{2}%
((-1,1);w_{-2,\beta})$ is given by%
\begin{equation}
\Delta:=\{f:(-1,1)\rightarrow\mathbb{C}\mid f,f^{\prime}\in AC_{\mathrm{loc}%
}(-1,1);f,m_{-2,\beta}[f]\in L^{2}((-1,1);w_{-2,\beta})\}\,.
\label{maximal domain - special}%
\end{equation}
For $f,g\in\Delta,$ Green's formula is%
\[
\int_{-1}^{1}m_{-2,\beta}[f](x)\overline{g}(x)w_{-2,\beta}(x)dx-\int_{-1}%
^{1}m_{-2,\beta}[\overline{g}](x)f(x)w_{-2,\beta}(x)dx=[f,g]_{w_{-2,\beta}%
}\big|_{x=-1}^{x=1}\,,
\]
where $\left[  \cdot,\cdot\right]  $ is the sesquilinear form defined by%
\[
\lbrack f,g]_{w_{-2,\beta}}(x):=(1-x)^{-1}(1+x)^{\beta+1}(f(x)\overline
{g}^{\prime}(x)-f^{\prime}(x)\overline{g}(x))\quad(x\in(-1,1);f,g\in\Delta)
\]
and
\[
\lbrack f,g]_{w_{-2,\beta}}(\pm1)=\lim_{x\rightarrow\pm1^{\mp}}%
[f,g]_{w_{-2,\beta}}(x)\quad(f,g\in\Delta).
\]
Moreover, for $f,g\in\Delta$ and $-1<x,y<1,$ Dirichlet's formula reads%
\begin{align}
&  \int_{x}^{y}m_{-2,\beta}[f](t)\overline{g}(t)w_{-2,\beta}(t)dt+(1-t)^{-1}%
(1+t)^{\beta+1}f^{\prime}(t)\overline{g}(t)\big|_{x}^{y}%
\label{Dirichlet's formula}\\
=  &  \int_{x}^{y}f^{\prime}(t)\overline{g}^{\prime}(t)(1-t)^{-1}%
(1+t)^{\beta+1}dt+\int_{x}^{y}f(t)\overline{g}(t)(1-t)^{-2}(1+t)^{\beta
}dt.\nonumber
\end{align}
The maximal operator $T_{\max}$ in $L^{2}((-1,1);w_{-2,\beta}),$ associated
with $m_{-2,\beta}[\cdot],$ is defined as
\begin{align*}
T_{\max}[f]  &  =m_{-2,\beta}[f]\\
f\in\mathcal{D}(T_{\max}) :  &  =\Delta
\end{align*}
and the minimal operator $T_{\min}$, the adjoint of $T_{\max},$ is given by%
\begin{align*}
T_{\min}[g]  &  =m_{-2,\beta}[g]\\
f\in\mathcal{D}(T_{\min}) :  &  =\{g\in\Delta\mid\lbrack g,f]_{w_{-2,\beta}%
}\big|_{x=-1}^{x=1}=0\text{ for all }f\in\Delta\}.
\end{align*}
The endpoints $x=\pm1$ are regular singular endpoints of $m_{-2,\beta}[\cdot]$
in the sense of Frobenius. Elementary calculations show that the Frobenius
indicial equations at, respectively, $x=1$ and $x=-1$ are $r(r-2)=0$ and
$r(r+\beta)=0.$ It follows that $m_{-2,\beta}[\cdot]$ is in the limit-point
case at $x=1$ while $m_{-2,\beta}[\cdot]$ is in the limit-circle case at
$x=-1$ when $-1<\beta<1$ and in the limit-point case at $x=-1$ when $\beta
\geq1.$ Applying the Glazman-Krein-Naimark theory, all self-adjoint operators
$S$ in $L^{2}((-1,1);w_{-2,\beta})$, generated by $m_{-2,\beta}[\cdot],$ have
the form
\[
S[f]=m_{-2,\beta}[f]
\]
for $f\in\mathcal{D}(S)$ where
\[
f\in\mathcal{D}(S):=\left\{
\begin{array}
[c]{ll}%
\Delta & \text{if }\beta\geq1\medskip\\
\{f\in\Delta\mid\lim_{x\rightarrow-1^{+}}[f,g_{S}](x)=0\} & \text{if }%
-1<\beta<1
\end{array}
\right.
\]
and where $g_{S}\in\Delta\setminus\mathcal{D}(T_{\min})$ (such a $g_{S}$ is
called a Glazman boundary function).

\section{A Certain Self-Adjoint Operator Generated by $m_{-2,\beta}[\cdot
]$\label{Section Seven}}

We are interested in the particular self-adjoint operator which has the Jacobi
polynomials $\left\{P_{n}^{(-2,\beta)}\right\}_{n=2}^{\infty}$ as eigenfunctions and has
spectrum $\{n^{2}+(\beta-1)n+1\mid n\geq2\}$.

Let $\widetilde{g}:[-1,1]\rightarrow\mathbb{R}$ be a twice continuously
differentiable function such that%
\begin{equation}
\widetilde{g}(x)=\left\{
\begin{array}
[c]{ll}%
1 & \text{if }x\text{ is near }-1\\
0 & \text{if }x\text{ is near }+1.
\end{array}
\right.  \label{Special g}%
\end{equation}
It is clear that $\widetilde{g}\in\Delta.$ We claim that there exists an
$\widetilde{f}$ $\in\Delta$ such that
\[
\lbrack\widetilde{g},\widetilde{f}]_{w_{-2,\beta}}(1)-[\widetilde{g}%
,\widetilde{f}]_{w_{-2,\beta}}(-1)\neq0.
\]
Of course, this would mean that $\widetilde{g}\notin\mathcal{D}(T_{\min}).$
Let $\widetilde{f}(x)=(1-x)^{2}(1+x)^{-\beta}.$ Remarkably,
\[
m_{-2,\beta}[(1-x)^{2}(1+x)^{-\beta}]=(-2\beta+2)(1-x)^{2}(1+x)^{-\beta}%
\]
and, since $-1<\beta<1,$ we see that $\widetilde{f}\in\Delta.$ Moreover, a
calculation shows that
\[
\lbrack\widetilde{g},\widetilde{f}]_{w_{-2,\beta}}(1)-[\widetilde{g}%
,\widetilde{f}]_{w_{-2,\beta}}(-1)=2\beta\neq0.
\]
Hence $\widetilde{g}(x)$ is a Glazman boundary function. Moreover, for
$f\in\Delta,$ observe that
\begin{align}
0  &  =-\lim_{x\rightarrow-1^{+}}[f,\widetilde{g}](x)\nonumber\\
\Longleftrightarrow0  &  =\lim_{x\rightarrow-1^{+}}(1-x)^{-1}(1+x)^{\beta
+1}(f^{\prime}(x)\widetilde{g}(x)-f(x)\widetilde{g}^{\prime}(x))\nonumber\\
\Longleftrightarrow0  &  =\lim_{x\rightarrow-1^{+}}(1+x)^{\beta+1}f^{\prime
}(x). \label{Simple BC}%
\end{align}
Furthermore, a calculation shows that, for $n\geq2,$
\[
\lim_{x\rightarrow-1^{+}}[P_{n}^{(-2,\beta)},\widetilde{g}](x)=-\lim
_{x\rightarrow-1^{+}}(1-x)^{-1}(1+x)^{\beta+1}(P_{n}^{(-2,\beta)}(x))^{\prime
}=0.
\]
Consequently, from (\ref{Simple BC}) and Theorem
\ref{Density of Jacobi - Special}, the following theorem is immediate from the
general Glazman-Krein-Naimark theory \cite{Naimark}.

\begin{theorem}
\label{Self-Adjointness of Special Jacobi Operator}Suppose $\beta>-1.$ The
operator
\[
T_{-2,\beta}:L^{2}((-1,1);w_{-2,\beta})\rightarrow L^{2}((-1,1);w_{-2,\beta})
\]
defined by%
\[
T_{-2,\beta}[f]=m_{-2,\beta}[f]
\]
for $f\in\mathcal{D}(T_{-2,\beta}),$ where%
\begin{equation}
\mathcal{D}(T_{-2,\beta}):=\left\{
\begin{array}
[c]{ll}%
\Delta & \text{if }\beta\geq1\medskip\\
\{f\in\Delta\mid\lim_{x\rightarrow-1^{+}}(1-x)^{\beta+1}f^{\prime}(x)=0\} &
\text{if }-1<\beta<1
\end{array}
\right.  \label{BC for Special}%
\end{equation}
is self-adjoint. Furthermore, the non-classical Jacobi polynomials
$\left\{P_{n}^{(-2,\beta)}\right\}_{n=2}^{\infty}$ form a complete orthogonal set of
eigenfunctions of $T_{-2,\beta}$ in $L^{2}((-1,1);w_{-2,\beta})$. The spectrum
$\sigma(T_{-2,\beta})$ is discrete and consists of the simple eigenvalues
\[
\sigma(T_{-2,\beta})=\sigma_{p}(T_{-2,\beta})=\{n^{2}+(\beta-1)n+1\mid
n\geq2\}.
\]

\end{theorem}

\begin{remark}
In a non-rigorous sense, the operator $T_{-2,\beta}$, given above in Theorem
\ref{Self-Adjointness of Special Jacobi Operator}, can be viewed as a `limit'
$($as $\alpha\rightarrow0)$ of the exceptional $X_{1}$-Jacobi self-adjoint
operator $\widehat{T}_{\alpha,\beta}$ given in Theorem \ref{t-abge1}; that is
to say,
\[
\lim_{\alpha\rightarrow0}\widehat{T}_{\alpha,\beta}=T_{-2,\beta}.
\]
Notice that the boundary conditions $($\ref{BC for Exceptional}$)$ and
$($\ref{BC for Special}$)$ for both operators coincide; however, there is one
significant difference. Indeed, the boundary condition given in $($%
\ref{BC for Exceptional}$)$ $($specifically the one when $0<\alpha<1$ and
$0<\beta<1)$ is determined using the Glazman boundary function $g(x)=1$ on
$(-1,1)$ while the boundary condition in $($\ref{BC for Special}$)$ is
determined using the Glazman boundary function $\widetilde{g}$ defined in
$($\ref{Special g}$).$ This latter function $\widetilde{g}$ is only $1$ near
$x=-1.$ In fact, we cannot use $g(x)\equiv1$ to obtain $T_{-2,\beta}$ since
this function does not belong to $L^{2}((-1,1);w_{-2,\beta}).$
\end{remark}

We now turn our attention to showing that $T_{-2,\beta}$ is a positive
operator in $L^{2}((-1,1);w_{-2,\beta});$ specifically, we prepare to show
that%
\begin{equation}
(T_{-2,\beta}f,f)_{w_{-2,\beta}}\geq(f,f)_{w_{-2,\beta}}\quad(f\in
\mathcal{D}(T_{-2,\beta})).
\label{Positivity of Special Self Adjoint Operator}%
\end{equation}
It is precisely this reason that we perturbed the Jacobi expression
$\widehat{\ell}_{0,\beta}[\cdot]$ in (\ref{Unperturbed Jacobi DE}) and shifted
our study to $m_{-2,\beta}[\cdot]$ in (\ref{JacobiDE-special}). Once we
establish (\ref{Positivity of Special Self Adjoint Operator}), then we can
apply the general left-definite theory of Littlejohn and Wellman
\cite{Littlejohn-Wellman} to construct a self-adjoint operator, generated by
$m_{-2,\beta}[\cdot]$, in the Sobolev space $S$ having inner product
$\phi(\cdot,\cdot),$ defined in (\ref{Sobolev IP}). We establish this
positivity (in Theorem \ref{positivity theorem} below) after proving two key
technical theorems (Theorems \ref{Properties at -1} and \ref{Properties at +1}%
), which concern the regularity, at the endpoints $x=\pm1,$ of functions from
the domain of $T_{-2,\beta}$ and from the maximal domain $\Delta.$

\begin{theorem}
\label{Properties at -1}Suppose $\beta>-1$ and $T_{-2,\beta}$ is the
self-adjoint operator defined in Theorem
\ref{Self-Adjointness of Special Jacobi Operator}. Let $f,g\in\mathcal{D}%
(T_{-2,\beta}).$ Then

\begin{itemize}
\item[(a)] $\lim_{x\rightarrow-1}(1+x)^{\beta+1}f^{\prime}(x)=0;$

\item[(b)] $\lim_{x\rightarrow-1}(1+x)^{(\beta+1)/2}f^{\prime}\in
L^{2}(-1,0);$

\item[(c)] $\lim_{x\rightarrow-1}(1-x)^{-1}(1+x)^{\beta+1}f^{\prime
}(x)\overline{g}(x)=0.$
\end{itemize}

\begin{proof}
\underline{(a)}: This limit is evident in the case $-1<\beta<1$ $($see
$($\ref{BC for Special}$))$ so suppose $\beta\geq1.$ Since $m_{-2,\beta}%
[\cdot]$ is in the limit-point case at $x=-1,$ the general Weyl theory $($see
\cite[Chapter 18]{Hellwig}$)$ states that%
\begin{equation}
\lim_{x\rightarrow-1^{+}}(1-x)^{-1}(1+x)^{\beta+1}(f^{\prime}(x)\overline
{g}(x)-f(x)\overline{g}^{\prime}(x))=0\quad(f,g\in\mathcal{D}(T_{-2,\beta})).
\label{Special Limit}%
\end{equation}
In particular, this limit is zero for all $f\in\mathcal{D}(T_{-2,\beta})$ and
the special choice $g$ defined by%
\[
g(x)=\left\{
\begin{array}
[c]{ll}%
1 & \text{if }-1\leq x\leq-1/2\\
16x^{3}+12x^{2} & \text{if }-1/2<x\leq0\\
0 & \text{if }0<x\leq1.
\end{array}
\right.
\]
A calculation shows that substitution of this $g$ into $($\ref{Special Limit}%
$)$ yields the required result.\medskip\newline\noindent\underline{(b)}:
Assume, without loss of generality, that $f$ is real-valued. For $-1<x\leq0$,
\begin{align}
&  \int_{x}^{0}m_{-2,\beta}[f](t)f(t)(1-t)^{-2}(1+t)^{\beta}dt-\int_{x}%
^{0}\left\vert f(t)\right\vert ^{2}(1-t)^{-2}(1+t)^{\beta+1}dt+f^{\prime
}(0)f(0)\label{Dirichlet Identity for Property at -1}\\
-  &  (1-x)^{-1}(1+x)^{\beta+1}f^{\prime}(x)f(x)=\int_{x}^{0}\left\vert
f^{\prime}(t)\right\vert ^{2}(1-t)^{-1}(1+t)^{\beta+1}dt.\nonumber
\end{align}
As $x\rightarrow-1^{+},$ the two integral terms on the left-hand side of
$($\ref{Dirichlet Identity for Property at -1}$)$ both converge and are
finite. If $(1+x)^{(\beta+1)/2}f^{\prime}\notin L^{2}(-1,0),$ then
\[
\int_{-1}^{0}\left\vert f^{\prime}(t)\right\vert ^{2}(1-t)^{-1}(1+t)^{\beta
+1}dt=\infty.
\]
It follows from $($\ref{Dirichlet Identity for Property at -1}$)$ that%
\[
\lim_{x\rightarrow-1^{+}}(1-x)^{-1}(1+x)^{\beta+1}f^{\prime}(x)f(x)=-\infty.
\]
Hence there exists $x^{\ast}\in(-1,0)$ such that $(1-x)^{-1}(1+x)^{\beta
+1}f^{\prime}(x)f(x)\leq-1$ for $x\in(-1,x^{\ast}]$. Notice, by part $(a),$
that $\left(  (1+x)^{\beta+1}f^{\prime}(x)\right)  ^{\prime}\neq0$ for $x\in$
$(-1,x^{\ast}].$ Without loss of generality, suppose that $(1+x)^{\beta
+1}f^{\prime}(x)>0$ and $f(x)<0$ for $x\in$ $(-1,x^{\ast}]$ so that
\begin{equation}
-\left\vert \left(  (1+x)^{\beta+1}f^{\prime}(x)\right)  ^{\prime}\right\vert
f(x)\geq\frac{\left\vert \left(  (1+x)^{\beta+1}f^{\prime}(x)\right)
^{\prime}\right\vert }{(1+x)^{\beta+1}f^{\prime}(x)}\quad(x\in(-1,x^{\ast}]).
\label{1- Properties at -1}%
\end{equation}
Then, for $-1<x\leq x^{\ast},$%
\begin{align*}
\infty &  >-\int_{-1}^{1}\left\vert m_{-2,\beta}[f](t)f(t)-f^{2}(t)\right\vert
(1-t)^{-2}(1+t)^{\beta+1}dt\\
&  =-\int_{-1}^{1}\left\vert \left(  (1-t)^{-1}(1+t)^{\beta+1}f^{\prime
}(t)\right)  ^{\prime}f(t)\right\vert dt\geq\int_{-1}^{1}\frac{\left\vert
\left(  (1+t)^{\beta+1}f^{\prime}(t)\right)  ^{\prime}\right\vert
}{(1+t)^{\beta+1}f^{\prime}(t)}dt\\
&  \geq\int_{x}^{x^{\ast}}\frac{\left\vert \left(  (1+t)^{\beta+1}f^{\prime
}(t)\right)  ^{\prime}\right\vert }{(1+t)^{\beta+1}f^{\prime}(t)}%
dt\geq\left\vert \int_{x}^{x^{\ast}}\frac{\left(  (1+t)^{\beta+1}f^{\prime
}(t)\right)  ^{\prime}}{(1+t)^{\beta+1}f^{\prime}(t)}dt\right\vert \\
&  =\left\vert \ln\left(  (1+t)^{\beta+1}f^{\prime}(t)\right)  \right\vert
_{x}^{x^{\ast}}=\left\vert K-\ln((1+x)^{\beta+1}f^{\prime}(x))\right\vert \\
&  \rightarrow\infty\text{ since }(1+x)^{\beta+1}f^{\prime}(x)\rightarrow
0\text{ as }x\rightarrow-1^{+}.
\end{align*}
This contradiction establishes part (b).\medskip\newline\noindent
\underline{(c)}: The argument to prove this result mirrors closely the above
proof of part (b). Assume both $f$ and $g$ are real-valued. From the identity%
\begin{align*}
&  \int_{x}^{0}m_{-2,\beta}[f](t)\overline{g}(t)(1-t)^{-2}(1+t)^{\beta}%
dt-\int_{x}^{0}f(t)\overline{g}(t)(1-t)^{-2}(1+t)^{\beta}dt+f^{\prime
}(0)\overline{g}(0)\\
&  =(1-x)^{-1}(1+x)^{\beta+1}f^{\prime}(x)\overline{g}(x)+\int_{x}%
^{0}f^{\prime}(t)\overline{g}^{\prime}(t)(1-t)^{-1}(1+t)^{\beta+1}dt,
\end{align*}
and the fact that each integral terms is finite as $x\rightarrow-1^{+},$ we
see that%
\[
\lim_{x\rightarrow-1^{+}}(1-x)^{-1}(1+x)^{\beta+1}f^{\prime}(x)g(x)
\]
exists and is finite. Suppose this limit equals $c;$ if $c\neq0,$ suppose,
without loss of generality, that $c>0$. Then there exists $x^{\ast}\in(-1,0)$
such that%
\[
(1-x)^{-1}(1+x)^{\beta+1}f^{\prime}(x)g(x)\geq\frac{c}{2}\quad(x\in
(-1,x^{\ast}])
\]
where, without loss of generality, $f^{\prime}(x)>0$ and $g(x)>0$ for
$x\in(-1,x^{\ast}]$. The case when $f^{\prime}(x)<0$ and $g(x)<0$ for
$x\in(-1,x^{\ast}]$ follows in analogy. Hence%
\[
g(x)\geq\frac{c}{2}\cdot\frac{1}{(1-x)^{-1}(1+x)^{\beta+1}f^{\prime}(x)}%
\quad(x\in(-1,x^{\ast}])
\]
so that%
\[
\left\vert \left(  (1-x)^{-1}(1+x)^{\beta+1}f^{\prime}(x)\right)  ^{\prime
}\right\vert g(x)\geq\frac{c}{2}\cdot\frac{\left\vert \left(  (1-x)^{-1}%
(1+x)^{\beta+1}f^{\prime}(x)\right)  ^{\prime}\right\vert }{(1-x)^{-1}%
(1+x)^{\beta+1}f^{\prime}(x)}\quad(x\in(-1,x^{\ast}]).
\]
Integrate to obtain%
\begin{align*}
\infty &  >\int_{-1}^{1}\left\vert m_{-2,\beta}[f](t)-f(t)\right\vert
g(t)(1-t)^{-2}(1+t)^{\beta}dt\\
&  =\int_{-1}^{1}\left\vert \left(  (1-t)^{-1}(1+t)^{\beta+1}f^{\prime
}(t)\right)  ^{\prime}\right\vert g(t)dt\\
&  \geq\frac{c}{2}\left\vert \int_{x}^{x^{\ast}}\frac{\left(  (1-t)^{-1}%
(1+t)^{\beta+1}f^{\prime}(t)\right)  ^{\prime}}{(1-t)^{-1}(1+t)^{\beta
+1}f^{\prime}(t)}dt\right\vert \\
&  =\frac{c}{2}\left\vert \ln\left(  (1-t)^{-1}(1+t)^{\beta+1}f^{\prime
}(t)\right)  \right\vert _{x}^{x^{\ast}}\rightarrow\infty\text{ as
}x\rightarrow-1^{+}.
\end{align*}
This contradiction shows that $c=0$ and establishes part $($c$).$ This
completes the proof of the theorem.
\end{proof}
\end{theorem}

Before we can prove Theorem \ref{Properties at +1}, which describes properties
of functions at $x=1$ in the maximal domain $\Delta,$ we need to recall an
`$L^{2}$ inequality' due to Chisholm, Everitt and Littlejohn (see \cite{CHEL}).

\begin{theorem}
\label{CHEL Inequality}Let $I=(a,b)$ where $-\infty\leq a<b\leq\infty$ and let
$\omega$ be a positive Lebesgue measurable function on $I$. Let $\varphi
,\psi:I\rightarrow\mathbb{C}$ satisfy:

\begin{itemize}
\item[(a)] $\varphi,\psi\in L_{\mathrm{loc}}^{2}(I;\omega);$

\item[(b)] There exists $c\in(a,b)$ such that $\varphi\in L^{2}((a,c];\omega)$
and $\psi\in L^{2}([c,b);\omega)$. (In this case, we say that `$\varphi$ is
$L^{2}$ near $a$' and `$\psi$ is $L^{2}$ near $b$', both with respect to the
weight $\omega$.);

\item[(c)] For all $[\delta,\gamma]\subset I,$
\[
\int_{a}^{\delta}|\varphi|^{2}\omega dx>0\text{ and }\int_{\gamma}^{b}%
|\psi|^{2}\omega dx>0.
\]

\end{itemize}

\noindent Define $A,B:L^{2}(I;\omega)\rightarrow L_{\mathrm{loc}}^{2}%
(I;\omega)$ by
\[
(Af)(x):=\varphi(x)\int_{x}^{b}\psi(x)f(x)\,\omega(x)dx\quad(\,x\in(a,b)\text{
and }f\in L^{2}(\,(a,b);\omega)\,)
\]%
\[
(Bg)(x):=\psi(x)\int_{a}^{x}\varphi(x)g(x)\omega(x)\,dx\quad(\,x\in(a,b)\text{
and }g\in L^{2}(\,(a,b);\omega)\,)
\]
and $K:(a,b)\rightarrow(0,\infty)$ by%
\[
K(x):=\left\{  \int_{a}^{x}\left\vert \varphi\right\vert ^{2}\omega\right\}
^{1/2}\left\{  \int_{x}^{b}\left\vert \psi\right\vert ^{2}\omega\right\}
^{1/2}\quad(\,x\in(a,b)\,)
\]
and the number $K\in(0,\infty]$ by%
\[
K:=\sup\{\,K(x):x\in(a,b)\,\}.
\]
Then a necessary and sufficient condition for both $A$ and $B$ to be bounded
linear operators into $L^{2}(I;\omega)$ is for $K$ to be finite. Furthermore,
in this case,
\[
||Af||_{\omega}\leq2K||f||_{\omega}\text{ and }||Bf||_{\omega}\leq
2K||f||_{\omega}\quad(f\in L^{2}(I;\omega)).
\]

\end{theorem}

\begin{theorem}
\label{Properties at +1}Suppose $\beta>-1$ and $T_{-2,\beta}$ is the
self-adjoint operator defined in Theorem
\ref{Self-Adjointness of Special Jacobi Operator}. Let $f,g\in\Delta$ $($see
$($\ref{maximal domain - special}$)).$ Then
\end{theorem}

\begin{itemize}
\item[(a)] $f^{\prime}\in L^{2}[0,1);$

\item[(b)] $\lim_{x\rightarrow1^{-}}f(x)$ exists and is finite and $f\in
AC_{\mathrm{loc}}(-1,1];$

\item[(c)] $f(1)=0;$

\item[(d)] $(1-x)^{-1}f^{\prime}\in L^{2}[0,1)\subset L^{1}[0,1];$

\item[(e)] $\lim_{x\rightarrow1^{-}}(1-x)^{-1}(1+x)^{\beta+1}f^{\prime
}(x)\overline{g}(x)=0.$
\end{itemize}

\begin{proof}
\underline{(a)}: For $0\leq x<1$, note that
\begin{align}
f^{\prime}(x)=  &  (1-x)(1+x)^{\beta-1}\int_{0}^{x}\frac{(1-t)^{-1}%
(1+t)^{\beta/2}}{(1-t)^{-1}(1+t)^{\beta/2}}\left(  (1-t)^{-1}(1+t)^{\beta
+1}f^{\prime}(t)\right)  ^{\prime}\,dt\label{Identity for Property at 1}\\
&  \quad\quad+(1-x)(1+x)^{-\beta-1}f^{\prime}(0)\,.\nonumber
\end{align}
Clearly the term $(1-x)(1+x)^{-\beta-1}f^{\prime}(0)\in L^{2}(0,1).$ By
definition of $\Delta,$ we see that
\[
(1-t)^{-1}(1+t)^{\beta/2}\left(  (1-t)^{-1}(1+t)^{\beta+1}f^{\prime
}(t)\right)  ^{\prime}\in L^{2}(0,1).
\]
We apply Theorem \ref{CHEL Inequality} using $\psi(x)=(1-x)(1+x)^{-\beta
-1},\varphi(x)=(1-x)(1+x)^{-\beta/2},$ $\omega(x)=1$ and $\left(  a,b\right)
=\left(  0,1\right)  $. We see that $\varphi$ is $L^{2}$ near $0$ and $\psi$
is $L^{2}$ near $1.$ Moreover, since%
\[
\int_{0}^{x}\left\vert \varphi\right\vert ^{2}dt\cdot\int_{x}^{1}\left\vert
\psi\right\vert ^{2}dt\quad(0\leq x\leq1)
\]
is bounded on $[0,1),$ we conclude from Theorem \ref{CHEL Inequality} that
$f^{\prime}\in L^{2}[0,1)$.

\noindent\underline{(b)}: Note that from part (a), $f^{\prime}$ $\in
L^{1}(0,1)$. For $0\leq x<1$, we can write $f(x)=f(0)+\int_{0}^{x}f^{\prime
}(t)\,dt$. Then, $\lim_{x\rightarrow1^{-}}f(x)$ exists and is finite. We
define $f(1):=\lim_{x\rightarrow1^{-}}f(x)$. In this case we see that $f\in
AC_{\mathrm{loc}}[0,1]$. Since $f\in\Delta,$ it follows that $f\in AC(-1,1]$.

\noindent\underline{(c)}: Suppose $f(1)\neq0$. Without loss of generality,
assume $f(x)$ is real valued and $f(1)=c>0$. This implies that there exists
$x^{\ast}\in(0,1)$ such that $f(x)\geq\frac{c}{2}$ for all $x\in\lbrack
x^{\ast},1]$. Then,
\begin{align*}
\infty &  >\int_{-1}^{1}\left\vert f(x)\right\vert ^{2}(1-x)^{-2}(1+x)^{\beta
}\,dx\geq\int_{x^{\ast}}^{1}\left\vert f(x)\right\vert ^{2}(1-x)^{-2}%
(1+x)^{\beta}\,dx\\
&  \geq\left(  \frac{c}{2}\right)  ^{2}\int_{x^{\ast}}^{1}(1-x)^{-2}%
(1+x)^{\beta}\,dx=\infty\,,
\end{align*}
a contradiction. Hence $f(1)=0$.

\noindent\underline{(d)}: Note from \eqref{Identity for Property at 1} that%
\begin{align*}
(1-x)^{-1}f^{\prime}(x)=  &  (1+x)^{\beta-1}\int_{0}^{x}\frac{(1-t)^{-1}%
(1+t)^{\beta/2}}{(1-t)^{-1}(1+t)^{\beta/2}}\left(  (1-t)^{-1}(1+t)^{\beta
+1}f^{\prime}(t)\right)  ^{\prime}\,dt\\
&  \quad\quad+(1+x)^{-\beta-1}f^{\prime}(0)\,.
\end{align*}
Applying Theorem \ref{CHEL Inequality} with the same $\varphi$ and $\psi,$
part (d) follows similarly to part (a). \newline\noindent\underline{(e)}:
Suppose that $f,g\in\Delta$ are both real-valued. The reader can check the
following variant of Dirichlet's formula: for $0\leq x\leq1,$ we have%
\begin{align}
&  \int_{0}^{x}f^{\prime}(t)g^{\prime}(t)(1-t)^{-1}(1+t)^{\beta+1}%
dt=-f^{\prime}(0)g(0)+(1-x)^{-1}(1+x)^{\beta+1}f^{\prime}%
(x)g(x)\label{Variant of Dirichlet's Formula}\\
&  +\int_{0}^{x}m_{-2,\beta}[f](t)g(t)(1-t)^{-2}(1+t)^{\beta}dt-\int_{0}%
^{x}f(t)g(t)(1-t)^{-2}(1+t)^{\beta}dt.\nonumber
\end{align}
By part $($d$),$ $(1-x)^{-1}f^{\prime}\in L^{2}[0,1];$ furthermore since $g\in
AC[0,1]$ and $(1+x)^{\beta+1}$ is bounded on $[0,1],$ we can say that
\[
(1-x)^{-1}(1+x)^{\beta+1}f^{\prime}g\in L^{1}(0,1).
\]
Similarly, from part $($a$),$ we see that
\[
(1-x)^{-1}(1+x)^{\beta+1}f^{\prime}g^{\prime}\in L^{1}(0,1).
\]
Consequently, we see that each of the integral terms in
\eqref{Variant of Dirichlet's Formula} converge as $x\rightarrow1^{-}.$ Hence,
from \eqref{Variant of Dirichlet's Formula},
\[
\lim_{x\rightarrow1^{-}}(1-x)^{-1}(1+x)^{\beta+1}f^{\prime}(x)g(x)
\]
exists and is finite. Suppose that this limit equals $c$ but $c\neq0;$ without
loss of generality, we can assume that $c>0.$ Then there exists $x^{\ast}%
\in\lbrack0,1)$ such that, without loss of generality, $f^{\prime}(x)>0$,
$g(x)>0$ and
\[
(1-x)^{-1}(1+x)^{\beta+1}f^{\prime}(x)\geq\frac{c}{2}\cdot\frac{1}{g(x)}%
\quad(x\in\lbrack x^{\ast},1))\,.
\]
But then%
\begin{align*}
\infty &  >\int_{0}^{1}(1-t)^{-1}(1+t)^{\beta+1}f^{\prime}(t)\left\vert
g^{\prime}(t)\right\vert dt\\
&  \geq\int_{x^{\ast}}^{1}(1-t)^{-1}(1+t)^{\beta+1}f^{\prime}(t)\left\vert
g^{\prime}(t)\right\vert dt\geq\frac{c}{2}\int_{x^{\ast}}^{1}\frac{\left\vert
g^{\prime}(t)\right\vert }{g(t)}dt\\
&  \geq\left\vert \int_{x^{\ast}}^{1}\frac{g^{\prime}(t)}{g(t)}dt\right\vert
=\frac{c}{2}\left\vert \ln(g(t))\right\vert _{x^{\ast}}^{1}=\infty\text{ since
}g(1)=0.
\end{align*}
This contradiction completes the proof of part $($d$)$ and the theorem.
\end{proof}

\begin{theorem}
\label{positivity theorem} For $f\in\mathcal{D}(T_{-2,\beta}),$ the positivity
inequality in $($\ref{Positivity of Special Self Adjoint Operator}$)$ holds;
that is to say $T_{-2,\beta}$ is bounded below in $L^{2}((-1,1);w_{-2,\beta})$
by the identity operator $I$ so
\[
(T_{-2,\beta}[f],f)_{w_{-2,\beta}}\geq(f,f)_{w_{-2,\beta}}\quad(f\in
\mathcal{D}(T_{-2,\beta})).
\]

\begin{proof}
Let $f\in\mathcal{D}(T_{-2,\beta}).$ Let $g=f$ in \eqref{Dirichlet's formula}
and let $x\rightarrow-1^{+}$ and $y\rightarrow1^{-}.$ From Property (c) in
Theorem \ref{Properties at -1} and Property (e) in Theorem
\ref{Properties at +1}, the result readily follows.
\end{proof}
\end{theorem}

In fact, if $f,g\in\mathcal{D}(T_{-2,\beta})$ and, in
\eqref{Dirichlet's formula}, we let $x\rightarrow-1^{+},y\rightarrow1^{-}$, we
obtain%
\begin{equation}
(T_{-2,\beta}f,g)_{w_{-2,\beta}}=\int_{-1}^{1}\left(  (1-t)^{-1}%
(1+t)^{\beta+1}f^{\prime}(t)\overline{g}^{\prime}(t)+(1-t)^{-2}(1+t)^{\beta
}f(t)\overline{g}(t)\right)  dt. \label{1st LD IP}%
\end{equation}
The right-hand side of (\ref{1st LD IP}) is an inner product; in fact, it is
the first left-definite inner product associated with the pair $(T_{-2,\beta
},L^{2}((-1,1);w_{-2,\beta}))$; see Section \ref{Section Eight} and
\cite{Littlejohn-Wellman} for further information.

Since $T_{-2,\beta}$ is positive, we can apply the left-definite theory that
Littlejohn and Wellman developed in \cite{Littlejohn-Wellman}. Without going
into detail, their general results show that $\mathcal{D}(T_{-2,\beta})$ is
equal to the second left-definite space $V_{2}$ associated with $(T
_{-2,\beta},L^{2}((-1,1),w_{-2,\beta})).$ More specifically,

\begin{theorem}
\label{LD Characterization of Domain}The domain of the operator $T_{-2,\beta}$
is given by%
\begin{align*}
\mathcal{D}(T_{-2,\beta})  &  =\left\{  f:(-1,1)\rightarrow\mathbb{C\mid
}f,f^{\prime}\in AC_{\mathrm{loc}}(-1,1);\right. \\
&  \quad\left.  f^{(j)}\in L^{2}((-1,1);(1-x)^{j-2}(1+x)^{\beta+j})\text{
}(j=0,1,2)\right\}  .
\end{align*}

\end{theorem}

In particular, for $f\in\mathcal{D}(T_{-2,\beta}),$ we see that $f^{\prime
\prime}\in L^{2}((-1,1);(1+x)^{\beta+2}).$ Since $(1+x)^{\beta+2}$ is bounded
on $[0,1],$ we then have $f^{\prime\prime}\in L^{2}[0,1].$ Hence, $f^{\prime
}\in AC[0,1];$ in fact,
\begin{equation}
f^{\prime}\in AC_{\mathrm{loc}}(-1,1]. \label{AC condition on f'}%
\end{equation}
In particular, $f^{\prime}(1)$ exists and is finite. In fact, it is necessary
that
\begin{equation}
f^{\prime}(1)=0. \label{Another Property of f in Domain}%
\end{equation}
For suppose $f^{\prime}(1)\neq0;$ without loss of generality, we can assume
that $f^{\prime}(1)=c>0.$ Hence, there exists $x^{\ast}\in(0,1)$ such that
$f^{\prime}(x)\geq\dfrac{c}{2}$ on $[x^{\ast},1).$ But then
\[
\int_{x^{\ast}}^{1}(1-t)^{-1}f^{\prime}(t)dt\geq\frac{c}{2}\int_{x^{\ast}}%
^{1}(1-t)^{-1}dt=\infty,
\]
contradicting part (d) of Theorem \ref{Properties at +1}. We note that
property (\ref{Another Property of f in Domain}) will be useful to us later in
this paper.

\section{A Primer on Left-Definite Operator Theory\label{Section Eight}}

Now that we have established that $T_{-2,\beta}$ is a self-adjoint operator
which is bounded below in $L^{2}((-1,1);w_{-2,\beta})$ by $I$ (see
(\ref{Positivity of Special Self Adjoint Operator})), we can apply the general
left-definite theory developed by Littlejohn and Wellman in
\cite{Littlejohn-Wellman}. This theory will be important as we continue our
study of $m_{-2,\beta}[\cdot]$ in the Sobolev space generated by the inner
product (\ref{Sobolev IP}). We now briefly discuss this theory.

Let $H$ be a Hilbert space with inner product $(\cdot,\cdot)$ and suppose
$A:\mathcal{D}(A)\subset H\rightarrow H$ is a self-adjoint operator that is
bounded below in $H$ by $kI$ for some $k>0;$ that is,%
\[
(Af,f)\geq k(f,f)\quad(f\in\mathcal{D}(A)).
\]
It follows that $A^{r}$ is self-adjoint and bounded below in $H$ by $k^{r}I$
for each $r>0.$

\begin{theorem}
\label{Left-definite Theorem}Suppose $r>0.$

\begin{itemize}
\item[(a)] Let
\begin{equation}
\left\{
\begin{array}
[c]{rl}%
V_{r} & =\mathcal{D}(A^{r/2})\\
(f,g)_{r} & =(A^{r/2}f,A^{r/2}g)\\
H_{r} & =(V_{r},(\cdot,\cdot)_{r}.
\end{array}
\right.  \label{LD Definitions}%
\end{equation}
Then%
\begin{equation}
\left\{
\begin{array}
[c]{cl}%
(\text{i}) & H_{r}\text{ is a Hilbert space;}\\
(\text{ii}) & \mathcal{D}(A^{r})\text{ is a subspace of }V_{r};\\
(\text{iii}) & \mathcal{D}(A^{r})\text{ is dense in }V_{r};\\
(\text{iv}) & \left(  f,f\right)  _{r}\geq k^{r}\left(  f,f\right)  \quad(f\in
V_{r});\\
(\text{v}) & \left(  f,g\right)  _{r}=\left(  A^{r}f,g\right)  \quad
(f\in\mathcal{D}(A^{r}),\;g\in V_{r}).
\end{array}
\right.  \label{Properties of LD space}%
\end{equation}

\item[(b)] The operator $A_{r}:\mathcal{D}(A_{r})\subset H_{r}\rightarrow
H_{r}$ given by%
\[
\left\{
\begin{array}
[c]{rl}%
A_{r}x & =Ax\\
x\in\mathcal{D}(A_{r}) & =V_{r+2}%
\end{array}
\right.
\]
is self-adjoint in $H_{r}$ and has spectrum $\sigma(A_{r})=\sigma(A)$ and is
bounded below in $H_{r}$ by $k^{r}I.$ Furthermore, if $\{\phi_{n}\}$ is a
complete set of eigenfunctions of $A$ in $H,$ then $\{\phi_{n}\}$ is a
complete set of eigenfunctions of $A_{r}$ in $H_{r}.$
\end{itemize}
\end{theorem}

The space $H_{r}$ is called the $r^{th}$ left-definite space associated with
the pair $(A,H).$ Notice, from (\ref{LD Definitions}) that $\mathcal{D}%
(A)=V_{2};$ this new characterization of the domain of $A$ has proven to be
useful in several applications. The operator $A_{r}$ is called the $r^{th}$
left-definite operator associated with $(A,H).$

The term `left-definite' owes its name to spectral theory of differential
operators. Indeed, if $A$ is self-adjoint, bounded below and generated by a
differential expression $\ell\lbrack\cdot]$, property (v) in
(\ref{Properties of LD space}) says that the study will be in the space whose
inner product is generated by the $r^{th}$ power $\ell^{r}[\cdot]$ of
$\ell\lbrack\cdot]$ which, of course, is on the\textit{ left }side of the
differential equation $\ell^{r}[y]=\lambda y.$

In our situation, it is not difficult to establish that the $r^{th}$
left-definite space, when $r\in\mathbb{N},$ associated with $(T_{-2,\beta
},L^{2}((-1,1);w_{-2,\beta}))$ is $H_{r}=(V_{r},(\cdot,\cdot)_{r})$, where%
\begin{align}
V_{r}  &  =\left\{  f:(-1,1)\rightarrow\mathbb{C}\right.  f,f^{\prime},\ldots,
f^{(r-1)}\in AC_{\mathrm{loc}}(-1,1);\label{V_r}\\
&  \left.  f^{(j)}\in L^{2}((-1,1);(1-x)^{j-2}(1+x)^{\beta+j}),\text{
}j=0,1,\ldots,r\right\}  ,\nonumber
\end{align}
and%
\begin{equation}
(f,g)_{r}=\sum_{j=0}^{n}c_{j}^{(-2,\beta)}(n)\int_{-1}^{1}f^{(j)}%
(x)\overline{g}^{(j)}(x)(1-x)^{j-2}(1+x)^{\beta+j}dx; \label{LD IP}%
\end{equation}
here, the numbers $\{c_{j}^{(-2,\beta)}\}$ are the so-called Jacobi-Stirling
numbers; see \cite{Andrews-Egge-Gawronski-Littlejohn} and \cite{EKLWY-Jacobi}.
When $r=2,$ the inner product in (\ref{LD IP}) is specifically given by%
\begin{align}
(f,g)_{2}  &  =\int_{-1}^{1}\left(  (1+t)^{\beta+2}f^{\prime\prime
}(t)\overline{g}^{\prime\prime}(t)+(\beta+2)(1-t)^{-1}(1+t)^{\beta+1}%
f^{\prime}(t)\overline{g}^{\prime}(t)\right. \label{V_2 IP}\\
&  \left.  +(1-t)^{-2}(1+t)^{\beta}f(t)\overline{g}(t)\right)  dt.\nonumber
\end{align}
Notice also when $r=2$ in (\ref{V_r})$,$ we obtain the characterization given
in Theorem \ref{LD Characterization of Domain}.

Another left-definite space which will be useful to us later in this paper is
\begin{align}
V_{4}  &  =\left\{  f:(-1,1)\rightarrow\mathbb{C\mid}f,f^{\prime}%
,f^{\prime\prime},f^{\prime\prime\prime}\in AC_{\mathrm{loc}}(-1,1);\right.
\label{V_4}\\
&  \left.  f^{(j)}\in L^{2}((-1,1);(1-x)^{j-2}(1+x)^{\beta+j},\text{
}j=0,1,2,3,4\right\}  .\nonumber
\end{align}
This space will turn out to be instrumental in constructing a certain
self-adjoint operator $T_{2}$ in the Sobolev space $S$ which we now introduce.

\section{The Sobolev space $\left(  S,\phi(\cdot,\cdot)\right)  $%
\label{Section Nine}}

Recall the Sobolev inner product $\phi(\cdot,\cdot)$ given in
(\ref{Sobolev IP}). The full sequence of non-classical Jacobi polynomials
$\left\{  P_{n}^{(-2,\beta)}\right\}  _{n=0}^{\infty},$ for $\beta>-1$ but
$\beta\neq0$ (see Remark \ref{beta not zero})$,$ are orthogonal with respect
to this inner product. Let
\[
S:=\left\{  f:(-1,1]\rightarrow\mathbb{C}\mid f,f^{\prime}\in AC_{\mathrm{loc}%
}(-1,1];\,f^{\prime\prime}\in L^{2}((-1,1);(1+x)^{\beta+2})\right\}  \,
\]
and let $\left\Vert \cdot\right\Vert _{\phi}$ be the usual norm associated
with $\phi(\cdot,\cdot)$; notice that%
\begin{equation}
\left\Vert f\right\Vert _{\phi}=\left\vert f^{\prime}(1)\right\vert
^{2}+\left\vert f(1)+\frac{2}{\beta}f^{\prime}(1)\right\vert ^{2}+\int%
_{-1}^{1}\left\vert f^{\prime\prime}(x)\right\vert ^{2}(1+x)^{\beta+2}%
dx\quad(f\in S). \label{phi norm}%
\end{equation}
We want to construct a self-adjoint operator $T,$ generated by $m_{-2,\beta
}[\cdot],$ in $S$ that has the Jacobi polynomials $\left\{  P_{n}^{(-2,\beta
)}\right\}  _{n=0}^{\infty}$ as eigenfunctions and has spectrum $\sigma
(T)=\{n^{2}+(\beta-1)n+1\mid n\in\mathbb{N}_{0}\}.$ Before we do this, we must
discuss certain properties of this Sobolev space $S.$

\begin{theorem}
The space $\left(  S,\phi(\cdot,\cdot)\right)  $ is a Hilbert space.
\end{theorem}

\begin{proof}
Suppose $\left\{  f_{n}\right\}  \subseteq S$ is a Cauchy sequence. Note that
\begin{align*}
\left\Vert f_{n}-f_{m}\right\Vert _{\phi}^{2}  &  =\phi(f_{n}-f_{m}%
,f_{n}-f_{m})\\
&  =\left\vert f_{n}(1)-f_{m}(1)+\frac{2}{\beta}\left(  f_{n}^{\prime
}(1)-f_{m}^{\prime}(1)\right)  \right\vert ^{2}\\
&  \quad+\left\vert f_{n}^{\prime}(1)-f_{m}^{\prime}(1)\right\vert ^{2}%
+\int_{-1}^{1}\left\vert f_{n}^{\prime\prime}(x)-f_{m}^{\prime\prime
}(x)\right\vert ^{2}(1+x)^{\beta+2}\,dx\,.
\end{align*}
From this identity, we see that $\left\{  f_{n}^{\prime\prime}\right\}
_{n=0}^{\infty}$ is Cauchy in $L^{2}\left(  (-1,1);(1+x)^{\beta+2}\right)  $,
and that the sequences $\left\{  f_{n}(1)\right\}  _{n=0}^{\infty}$ and
$\left\{  f_{n}^{\prime}(1)\right\}  _{n=0}^{\infty}$ are both Cauchy in
$\mathbb{C}$. Therefore, from the completeness of the spaces $L^{2}\left(
(-1,1);(1+x)^{\beta+2}\right)  $ and $\mathbb{C}$, there exists a function
$g\in L^{2}((-1,1;\,(1+x)^{\beta+2})$ and scalars $a,b\in\mathbb{C}$ such that
$\left\{  f_{n}^{\prime\prime}\right\}  _{n=0}^{\infty}$ converges to $g$ in
$L^{2}((-1,1;\,(1+x)^{\beta+2})$, $\left\{  f_{n}^{\prime}(1)\right\}
_{n=0}^{\infty}$ converges to $a$ in $\mathbb{C}$, and $\left\{
f_{n}(1)\right\}  _{n=0}^{\infty}$ converges to $b$ in $\mathbb{C}$.

Define the function $f:(-1,1]\rightarrow\mathbb{C}$ by
\[
f(x):=ax+(b-a)+\int_{x}^{1}\int_{t}^{1}g(u)\,du\,dt
\]
Then $f,f^{\prime}\in AC(-1,1]$ with $f(1)=b$ and $f^{\prime}(1)=a.$ Moreover,
$f^{\prime\prime}=g\in L^{2}((-1,1);(1+x)^{\beta+2})$ so $f\in S.$ Moreover,
it is straightforward to see that
\[
\left\Vert f_{n}-f\right\Vert _{\phi}^{2}\rightarrow0\!\quad(n\rightarrow
\infty),
\]
completing the proof of the theorem.
\end{proof}

\begin{theorem}
\label{dense} The set $\mathcal{P}$ of all polynomials is dense in
$(S,\phi(\cdot,\cdot))$. Equivalently, the Jacobi polynomials $\left\{
P_{n}^{(-2,\beta)}\right\}  _{n=0}^{\infty}$ form a complete orthogonal set in
$S.$
\end{theorem}

\begin{proof}
Let $f\in S$. Then, $f^{\prime\prime}\in L^{2}((-1,1);(1+x)^{\beta+2})$. Since
$\mathcal{P}$ is dense in $L^{2}((-1,1);\,(1+x)^{\beta+2})$, there exists
$p\in\mathcal{P}$ such that
\begin{equation}
\int_{-1}^{1}\left\vert f^{\prime\prime}(x)-p(x)\right\vert ^{2}%
(1+\beta)^{\beta+2}\,dx<\epsilon^{2}. \label{Density 1}%
\end{equation}
With the polynomial $q$ defined by
\[
q(x):=f^{\prime}(1)x+(f(1)-f^{\prime}(1))+\int_{x}^{1}\int_{t}^{1}%
p(u)\,du\,dt,
\]
we see that $f(1)=q(1)$, $f^{\prime}(1)=q^{\prime}(1)$. Moreover, by
\eqref{Density 1} we see that
\begin{align*}
\left\Vert f-q\right\Vert _{\phi}^{2}  &  =\left\vert f(1)-q(1)+\frac{2}%
{\beta}(f^{\prime}(1)-q^{\prime}(1))\right\vert ^{2}\\
&  \quad\quad+\left\vert f^{\prime}(1)-q^{\prime}(1)\right\vert ^{2}+\int%
_{-1}^{1}\left\vert f^{\prime\prime}(x)-q^{\prime\prime}(x)\right\vert
^{2}(1+x)^{\beta+2}\,dx\\
&  =\int_{-1}^{1}\left\vert f^{\prime\prime}(x)-q^{\prime\prime}(x)\right\vert
^{2}(1+x)^{\beta+2}\,dx<\varepsilon^{2}.
\end{align*}
This completes the proof of the theorem.
\end{proof}

For reasons that will be made clearer shortly, we now define two subspaces of
$S.$
\begin{align*}
S_{1}  &  :=\mathrm{span}\left\{  P_{0}^{(-2,\beta)},P_{1}^{(-2,\beta
)}\right\}  =\left\{  f\in S\mid f^{\prime\prime}(x)=0\right\}  \,,\text{
and}\\
S_{2}  &  :=\mathrm{span}\left\{  P_{n}^{(-2,\beta)}\right\}  _{n=2}^{\infty
}=\left\{  f\in S\mid f(1)=f^{\prime}(1)=0\right\}  \,.
\end{align*}

\begin{theorem}
$S=S_{1}\oplus S_{2}.$

\begin{proof}
Let $f\in S.$ We can write $f(x)$ as
\[
f(x)=g_{1}(x)+g_{2}(x),
\]
where
\[
g_{1}(x)=f^{\prime}(1)x+f(1)-f^{\prime}(1)\text{ and }g_{2}(x)=f(x)-f^{\prime
}(1)x-f(1)+f^{\prime}(1).
\]
It is clear that $g_{i}\in S_{i}$ for $i=1,2$ so $S=S_{1}+S_{2}.$ To show that
$S_{1}\perp S_{2},$ suppose $f_{1}\in S_{1}$ and $f_{2}\in S_{2}$ so
$f_{2}(1)=f_{2}^{\prime}(1)=0$ and $f_{1}^{\prime\prime}(x)=0.$ Then%
\begin{align*}
\phi(f_{1},f_{2})  &  =f_{1}(1)\overline{f_{2}}(1)+\frac{2}{\beta}\left(
f_{1}^{\prime}(1)\overline{f_{2}}(1)+f_{1}(1)\overline{f_{2}}^{\prime
}(1)\right)  +\left(  1+\frac{4}{\beta^{2}}\right)  f_{1}^{\prime}%
(1)\overline{f_{2}}^{\prime}(1)\\
&  \quad\quad+\int_{-1}^{1}f_{1}^{\prime\prime}(x)\overline{f_{2}}%
^{\prime\prime}(x)(1+x)^{\beta+2}dx\\
&  =0.
\end{align*}
This completes the proof of the theorem.
\end{proof}
\end{theorem}

We remark that, since $S_{1}$ and $S_{2}$ are closed subspaces of $S,$ both
$(S_{1},\phi(\cdot,\cdot))$ and $(S_{2},\phi(\cdot,\cdot))$ are Hilbert spaces.

In order to construct the self-adjoint operator $T$ in $S$, we will construct
two self-adjoint operators $T_{1}$ and $T_{2},$ both generated by
$m_{-2,\beta}[\cdot]$, in $S_{1}$ and $S_{2}$ respectively. The operator
$T=T_{1}\oplus T_{2},$ the direct sum of $T_{1}$ and $T_{2},$ will be the
self-adjoint operator in $S$ that has the properties we desire.

\section{The Construction of the Operators $T_{1}$, $T_{2}$ and $T$%
\label{Section Ten}}

Define $T_{1}:\mathcal{D}(T_{1})\subset S_{1}\rightarrow S_{1}$ by
\begin{align}
T_{1}f  &  =m_{-2,\beta}[f]\nonumber\\
f\in\mathcal{D}(T_{1}):  &  =S_{1}. \label{Operator T_1}%
\end{align}

It is straight forward to show that $T_{1}$ is symmetric with respect to the
inner product $\phi(\cdot,\cdot)$ and, since $S_{1}$ is two-dimensional, it
follows that $T_{1}$ is self-adjoint in $S_{1}.$ Moreover, it is clear that
\[
\sigma(T_{1})=\{n^{2}+(\beta-1)n+1\mid n=0,1\}.
\]
We now focus our attention on the construction of $T_{2}.$ It is remarkable
that the left-definite theory associated with $T_{-2,\beta}$ plays a very
significant role in this construction.

\begin{theorem}
$S_{2}=\mathcal{D}(T_{-2,\beta})=V_{2}.$
\end{theorem}

\begin{proof}
We already know from Theorem \ref{Left-definite Theorem} that $\mathcal{D}%
(T_{-2,\beta})=V_{2}.\medskip$\newline\underline{$S_{2}\subset V_{2}$}: Let
$f\in S_{2}.$ In particular, we see that $f^{\prime\prime}\in L^{2}%
((-1,1);(1+x)^{\beta+2})$ which, since $(1+x)^{\beta+2}$ is bounded on $[0,1]$
implies that $f^{\prime\prime}\in L^{2}[0,1].$ Since $f^{\prime}(1)=0,$ we see
that%
\[
f^{\prime}(x)=-\int_{x}^{1}f^{\prime\prime}(t)dt
\]
and%
\[
(1-x)^{-1/2}(1+x)^{(\beta+1)/2}f^{\prime}(x)=-(1-x)^{-1/2}(1+x)^{(\beta
+1)/2}\int_{x}^{1}f^{\prime\prime}(t)dt.
\]
We now use Theorem \ref{CHEL Inequality} to show that $f^{\prime}\in
L^{2}((-1,1);(1-x)^{-1}(1+x)^{\beta+1}).$ Let $\varphi(x)=1,\psi
(x)=-(1+x)^{-1/2}(1+x)^{(\beta+1)/2};$ both of these functions satisfy the
conditions of Theorem \ref{CHEL Inequality}. Since%
\[
\int_{x}^{1}1dt\cdot\int_{-1}^{x}(1+x)^{-1}(1+t)^{\beta+1}dt
\]
is bounded on $(-1,1)$, we can conclude from Theorem \ref{CHEL Inequality}
that $f^{\prime}\in L^{2}((-1,1);(1-x)^{-1}(1+x)^{\beta+1}).$ A similar
application of Theorem \ref{CHEL Inequality} shows that
\[
f\in L^{2}((-1,1);(1-x)^{-2}(1+x)^{\beta}).
\]
This proves that $S_{2}\subset V_{2}$.

\noindent\underline{$V_{2}\subset S_{2}$}: Let $f\in V_{2}=\mathcal{D}%
(T_{-2,\beta}).$ From part $($b$)$ of Theorem \ref{Properties at +1} and
$($\ref{AC condition on f'}$),$ we see that $f,f^{\prime}\in AC_{\mathrm{loc}%
}(-1,1].$ From Theorem \ref{LD Characterization of Domain}, we find that
$f^{\prime\prime}\in L^{2}((-1,1);(1+x)^{\beta+2}).$ Finally, part $($c$)$ of
Theorem \ref{Properties at +1} and $($\ref{Another Property of f in Domain}%
$),$ we see that $f(1)=f^{\prime}(1)=0.$ This shows $V_{2}\subset S_{2}$ and
completes the proof of the theorem.
\end{proof}

Now that we have established that $S_{2}=V_{2}$, what can we say about the
inner products $\phi(\cdot,\cdot)$ and $(\cdot,\cdot)_{2},$ where
$(\cdot,\cdot)_{2}$ is the second left-definite inner product defined in
\eqref{V_2 IP}? Remarkably, the answer is given in the following theorem.

\begin{theorem}
The inner products $\phi(\cdot,\cdot)$ and $(\cdot,\cdot)_{2}$ are equivalent
on the Hilbert spaces $(S_{2},\phi(\cdot,\cdot))$ and $(V_{2}=S_{2}%
,(\cdot,\cdot)_{2})$.
\end{theorem}

\begin{proof}
Let $f\in S_{2}=V_{2}.$ Then%
\begin{align*}
\left\Vert f\right\Vert _{2}^{2}=(f,f)_{2}  &  =\int_{-1}^{1}\left(
(1+t)^{\beta+2}\left\vert f^{\prime\prime}(t)\right\vert ^{2}+(\beta
+2)(1-t)^{-1}(1+t)^{\beta+1}\left\vert f^{\prime}(t)\right\vert ^{2}\right. \\
&  \quad\quad\left.  +(1-t)^{-2}(1+t)^{\beta}\left\vert f(t)\right\vert
^{2}\right)  dt\\
&  \geq\int_{-1}^{1}(1+t)^{\beta+2}\left\vert f(t)\right\vert ^{2}%
dt=\left\Vert f\right\Vert _{\phi}^{2}\,,
\end{align*}
where the latter identity follows from $($\ref{phi norm}$)$ and by definition
of $S_{2}.$ The Open Mapping Theorem (see \cite[Chapter 4.12, Problem
9]{Kreyszig}) now implies that the two inner products are equivalent.
\end{proof}

We now are in position to construct the self-adjoint operator $T_{2}$ in
$S_{2}$ which has the Jacobi polynomials $\left\{  P_{n}^{(-2,\beta)}\right\}
_{n=2}^{\infty}$ as eigenfunctions. Define $T_{2}:\mathcal{D}(T_{2})\subset
S_{2}\rightarrow S_{2}$ by%

\begin{align}
T_{2}f  &  =m_{-2,\beta}[f]\nonumber\\
f\in\mathcal{D}(T_{2})  &  =V_{4} \label{T_2 definition}%
\end{align}
where $V_{4}$ is defined in (\ref{V_4}). Notice that, other than the inner
product being different (but equivalent), this operator is essentially the
second left-definite operator associated with the pair $(T_{-2,\beta}%
,L^{2}((-1,1);w_{-2,\beta}))$. From the left-definite theory
\cite{Littlejohn-Wellman}, $T_{2}$ is self-adjoint in the second left-definite
space $H_{2}=(V_{2},(\cdot,\cdot)_{2})$. However, it requires work to show it
is self-adjoint in $S_{2}.$

\begin{theorem}
\label{selfadjoint} Let $T_{2}$ be the operator defined in \eqref{T_2 definition}.

\begin{itemize}
\item[(a)] $T_{2}$ is densely defined and closed in $\left(  S_{2},\phi
(\cdot,\cdot)\right)  ;$

\item[(b)] $T_{2}$ is symmetric in $\left(  S_{2},\phi(\cdot,\cdot)\right)  ;$
that is,
\[
\phi(T_{2}f,g)=\phi(f,T_{2}g)\quad(f,g\in\mathcal{D}(T_{2}));
\]

\item[(c)] $T_{2}$ is self-adjoint in $(S_{2},\phi(\cdot,\cdot))$ and has the
Jacobi polynomials$\left\{  P_{n}^{(-2,\beta)}\right\}  _{n=2}^{\infty}$ as
eigenfunctions. The spectrum of $T_{2}$ is $\sigma(T_{2})=\{n^{2}%
+(\beta-1)n+1\mid n\geq2\}.$
\end{itemize}

\begin{proof}
Somewhat surprisingly, the difficult part of the proof is in establishing the
symmetry, not the self-adjointness, of $T_{2}$. The domain $\mathcal{D}%
(T_{2})$ being dense in $S_{2}$ follows by direct analysis or from part
$($b$)$ of Theorem \ref{Left-definite Theorem} since $\left\{  P_{n}%
^{(-2,\beta)}\right\}  _{n=2}^{\infty}\subset$ $\mathcal{D}(T_{2}).$ We now
show that $T_{2}$ is closed in $S_{2}.$ Suppose $\{f_{n}\}\subset
\mathcal{D}(T_{2})$ with%
\begin{align*}
f_{n}  &  \rightarrow f\text{ in }(S_{2},\phi(\cdot,\cdot))\text{ and}\\
T_{2}f  &  \rightarrow g\text{ in }(S_{2},\phi(\cdot,\cdot)).\text{ }%
\end{align*}
We need to show that $f\in\mathcal{D}(T_{2})$ and $T_{2}f=g.$ Since equivalent
inner products have the same convergent sequences, it is clear that
$f_{n}\rightarrow f$ and $T_{2}f_{n}\rightarrow g$ in $H_{2}=(V_{2}%
,(\cdot,\cdot)).$ Moreover, since self-adjoint operators are closed and since
$T_{2}$ is self-adjoint in $H_{2},$ we see that $T_{2}$ is closed. Hence $f\in
V_{4}=\mathcal{D}(T_{2})$ and $T_{2}f=g.$ This proves part (a). Since a
closed, symmetric operator having a complete set of eigenfunctions in a
Hilbert space is self-adjoint $($see \cite[Theorem 3, page 173 and Theorem 6,
page 184]{Hellwig}$)$, we see that the self-adjointness of $T_{2}$ will follow
as soon as we establish the symmetry of $T_{2}$. To that end, for
$f,g\in\mathcal{D}(T_{2}),$ a laborious calculation shows that%
\begin{equation}
\phi(T_{2}f,g)-\phi(f,T_{2}g)=[f,g]_{\phi}(1)-[f,g]_{\phi}(-1),
\label{Green's formula for phi}%
\end{equation}
where $[\cdot,\cdot]_{\phi}$ is the sesquilinear form given by
\[
\lbrack f,g]_{\phi}(x)=(1-x)(1+x)^{\beta+3}(f^{\prime\prime\prime}%
(x)\overline{g}^{\prime\prime}(x)-f^{\prime\prime}(x)\overline{g}%
^{\prime\prime\prime}(x))\quad(f,g\in\mathcal{D}(T_{2}))
\]
and
\[
\lbrack f,g]_{\phi}(\pm1)=\lim_{x\rightarrow\pm1^{\mp}}[f,g]_{\phi}(x).
\]
These limits both exist and are finite by definition of $\mathcal{D}(T_{2})$.
We will show that, in fact,%
\[
\lbrack f,g]_{\phi}(\pm1)=0\quad(f,g\in\mathcal{D}(T_{2})).
\]
We show the details at $x=+1.\medskip$\newline\underline{Claim \#1}:
$\lim_{x\rightarrow1^{-}}(1-x)f^{\prime\prime\prime}(x)=0$ for $f\in
\mathcal{D}(T_{2}).$\newline Without loss of generality, suppose $f$ is
real-valued. Let
\[
\widehat{g}(x)=\left\{
\begin{array}
[c]{ll}%
0 & \text{if }-1\leq x\leq0\medskip\\
-704x^{7}+1208x^{6}-707x^{5}+140x^{4} & \text{if }0<x\leq\frac{1}{2}\medskip\\
\frac{1}{2}(x-1)^{2} & \text{if }\frac{1}{2}<x\leq1.
\end{array}
\right.
\]
It is clear that $\widehat{g}\in V_{4}.$ A calculation shows that
\[
\lim_{x\rightarrow1^{-}}[f,\widehat{g}]_{\phi}(x)=\lim_{x\rightarrow1^{-}%
}(1-x)(1+x)^{\beta+3}f^{\prime\prime\prime}(x).
\]
Consequently, we see that
\[
\lim_{x\rightarrow1^{-}}(1-x)f^{\prime\prime\prime}(x)\quad(f\in
\mathcal{D}(T_{2})).
\]
exists and is finite. If this limit, say $c,$ is not zero, we can suppose that
$c>0.$ Hence there exists $x^{\ast}\in(0,1)$ such that%
\[
f^{\prime\prime\prime}(x)\geq\frac{c}{2}\cdot\frac{1}{1-x}\quad(x\in\lbrack
x^{\ast},1))
\]
and consequently%
\[
(1-x)^{1/2}(1+x)^{(\beta+3)/2}f^{\prime\prime\prime}(x)\geq\frac{c}{2}%
\cdot\frac{(1+x)^{(\beta+3)/2}}{(1-x)^{1/2}}\quad(x\in\lbrack x^{\ast},1)).
\]
Since $f^{\prime\prime\prime}\in L^{2}((-1,1);(1-x)(1+x)^{\beta+3}),$ we see
that%
\[
\infty>\int_{x^{\ast}}^{1}\left\vert f^{\prime\prime\prime}(x)\right\vert
^{2}(1-x)(1+x)^{\beta+3}dx\geq\frac{c^{2}}{4}\int_{x^{\ast}}^{1}%
\frac{(1+x)^{\beta+3}}{1-x}dx=\infty.
\]
This contradiction proves the claim.\medskip\newline\underline{Claim\ \#2}:
$f^{\prime\prime\prime}\in L^{2}(0,1);$ consequently, $f,f^{\prime}%
,f^{\prime\prime}\in AC[0,1]$ and, in particular, $f^{\prime\prime}(1)$ exists
and is finite. To see this, note that%
\[
f^{\prime\prime\prime}(x)=f^{\prime\prime\prime}(0)+\int_{0}^{x}\frac
{f^{(4)}(t)(1-t)(1+t)^{(\beta+4)/4}}{(1-t)(1+t)^{(\beta+4)/4}}dt\quad(0\leq
x<1).
\]
By definition of $V_{4},$ $(1-t)(1+t)^{(\beta+4)/2}f^{(4)}\in L^{2}(-1,1).$ We
apply Theorem \ref{CHEL Inequality} using
\[
\varphi(x)=\frac{1}{(1-x)(1+x)^{(\beta+4)/4}}\text{ and }\psi(x)=1\quad(0\leq
x<1).
\]
A calculation shows that%
\[
\int_{0}^{x}\varphi^{2}(t)dt\cdot\int_{x}^{1}\psi^{2}(t)dt
\]
is bounded on $[0,1].$ Thus, it follows that $f^{\prime\prime\prime}\in
L^{2}(0,1).\medskip$\newline\underline{Claim \#3:} For all $f,g\in
\mathcal{D}(T_{2}),$%
\[
\lim_{x\rightarrow1^{-}}[f,g]_{\phi}(x)=0.
\]
It suffices to show that
\[
\lim_{x\rightarrow1^{-}}(1-x)(1+x)^{\beta+3}f^{\prime\prime\prime}%
(x)\overline{g}^{\prime\prime}(x)=0.
\]
We apply Claims 1 and 2 and see that%
\begin{align*}
&  \lim_{x\rightarrow1^{-}}(1-x)(1+x)^{\beta+3}f^{\prime\prime\prime
}(x)\overline{g}^{\prime\prime}(x)\\
=  &  \lim_{x\rightarrow1^{-}}(1+x)^{\beta+3}\cdot\lim_{x\rightarrow1^{-}%
}(1-x)f^{\prime\prime\prime}(x)\cdot\lim_{x\rightarrow1^{-}}\overline
{g}^{\prime\prime}(x)\\
=  &  0.
\end{align*}

A similar analysis shows
\[
\lim_{x\rightarrow-1^{+}}[f,g]_{\phi}(x)=0\quad(f,g\in\mathcal{D}(T_{2})).
\]
Referring to $($\ref{Green's formula for phi}$),$ we see that $T_{2}$ is
symmetric in $S_{2}$ and this completes the proof of the theorem.
\end{proof}
\end{theorem}

The following theorem was shown in \cite[Theorem 11.1]{ELW}.

\begin{theorem}
Suppose $H$ is a Hilbert space with the orthogonal decomposition
\[
H=H_{1}\oplus H_{2},
\]
where $H_{1}$ and $H_{2}$ are closed subspaces of $H.$ Suppose $A_{1}%
:\mathcal{D}(A_{1})\subset H_{1}\rightarrow H_{1}$ and $A_{2}:\mathcal{D}%
(A_{2})\subset H_{2}\rightarrow H_{2}$ are self-adjoint operators in $H_{1}$
and $H_{2},$ respectively. For $f_{1}\in\mathcal{D}(A_{1})$ and $f_{2}%
\in\mathcal{D}(A_{2})$, write%
\[
f=f_{1}+f_{2},
\]
and let $A=A_{1}\oplus A_{2}:\mathcal{D}(A)\subset H\rightarrow H$ be the
operator defined by%
\begin{align*}
Af  &  =A_{1}f_{1}+A_{2}f_{2}\\
f\in\mathcal{D}(A) :  &  =\mathcal{D}(A_{1})\oplus\mathcal{D}(A_{2}).
\end{align*}
Then $A$ is self-adjoint in $H.$
\end{theorem}

We remark that if these operators $A_{1}$ and $A_{2}$ are both generated by,
say, a linear differential expression $m[\cdot]$, then so is $A=A_{1}\oplus
A_{2}.$ Indeed, if $f=f_{1}+f_{2}\in\mathcal{D}(A_{1})\oplus\mathcal{D}%
(A_{2}),$ then
\[
Af=A_{1}f_{1}+A_{2}f_{2}=m[f_{1}]+m[f_{2}]=m[f_{1}+f_{2}]=m[f].
\]
We are now in position to state the main result of this section.

\begin{theorem}
\label{Self-Adjoint Operator in S} Let $T_{1}$ and $T_{2}$ be the self-adjoint
operators defined in, respectively, \eqref{Operator T_1} and
\eqref{T_2 definition}. Let $T:\mathcal{D}(T)\subset S\rightarrow S$ be the
operator defined by%
\begin{align*}
T  &  =T_{1}\oplus T_{2}\\
\mathcal{D}(T):  &  =\mathcal{D}(T_{1})\oplus\mathcal{D}(T_{1}).
\end{align*}
Then $T$ is a self-adjoint operator, generated by the Jacobi differential
expression $m_{-2,\beta}[\cdot],$ in the Sobolev space $S.$ The Jacobi
polynomials $\left\{  P_{n}^{(-2,\beta)}\right\}  _{n=0}^{\infty}$ form a
complete set of eigenfunctions of $T.$ The spectrum of $T$ is discrete and
given specifically by $\sigma(T)=$ $\{n^{2}+(\beta-1)n+1\mid n\in
\mathbb{N}_{0}\}.$\hfill
\end{theorem}


\begin{thebibliography}{99}                                                                                               %


\bibitem {Akhiezer-Glazman}N. I. Akhiezer and I. M. Glazman, \emph{Theory of
Linear Operators in Hilbert Space}, Dover Publications, New York, 1993.

\bibitem {Andrews-Egge-Gawronski-Littlejohn}G. E. Andrews, E. Egge, W.
Gawronski, and L. L. Littlejohn, \emph{The Jacobi-Stirling Numbers}, 120
(2013), 288--303.

\bibitem {X1-Laguerre}M. Atia, L. L. Littlejohn, and J. Stewart, \emph{ The
Spectral Theory of the $X_{1}$-Laguerre Polynomials}, Adv. Dyn. Syst. Appl. 8
(2) (2013), 181--192.

\bibitem {Bochner}S. Bochner, \emph{\"{U}ber Sturm-Liouvillesche
Polynomsysteme}, Math. Z. 29 (1929), no. 1, 730--736.

\bibitem {Chihara}T. S. Chihara, \emph{An Introduction to Orthogonal
Polynomials, }Gordon and Breach Publishers, New York, 1978.

\bibitem {CHEL}R. S. Chisholm, W. N. Everitt, and L. L. Littlejohn, \emph{ An
Integral Operator with Applications}, J. Inequal. Appl. 3 (1999), 245--266.

\bibitem {DS}N. Dunford and A. Schwartz, \emph{Linear Operators}, Part II,
Spectral Theory, John Wiley \& Sons, New York, 1963.

\bibitem {Everitt}W. N. Everitt, \emph{Note on the X1-Jacobi orthogonal
polynomials}, arXiv:0812.0728, 2008 [math.CA].

\bibitem {EKLWY-Jacobi}W. N. Everitt, K. H. Kwon, L. L. Littlejohn, R.
Wellman, and G. J. Yoon, \emph{Jacobi-Stirling numbers, Jacobi polynomials,
and the left-definite analysis of the classical Jacobi differential
expression,} J. Comput. Appl. Math., 208 (2007), 29-56.

\bibitem {ELW}W. N. Everitt, L. L. Littlejohn, and R. Wellman, \emph{The
Sobolev orthogonality and spectral analysis of the Laguerre polynomials
}$\{L_{n}^{-k}\}$\emph{ for positive integers k}\textit{, }J. Comput. Appl.
Math.\textbf{, }171 (2004), 199--234.

\bibitem {KMUG}D. G\'{o}mez-Ullate, N. Kamran, and R. Milson, \emph{An
extended class of orthogonal polynomials defined by a Sturm-Liouville
problem}, J. Math. Anal. Appl. 359 (2009), no. 1, 352--367.

\bibitem {KMUG1}D. G\'{o}mez-Ullate, N. Kamran, and R. Milson, \emph{An
extension of Bochner's problem: exceptional invariant subspaces}, J. Approx.
Theory 162 (2010), no. 5, 987--1006.

\bibitem {KMUG2}D. G\'{o}mez-Ullate, N. Kamran, and R. Milson,
\emph{Exceptional orthogonal polynomials and the Darboux transformation,} J.
Phys. A, 43 (2010), 434016.

\bibitem {KMUG3}D. G\'{o}mez-Ullate, N. Kamran, and R. Milson, \emph{Two-step
Darboux transformations and exceptional Laguerre polynomials}, J. Math. Anal.
Appl., 387 (2012), no. 1, 410--418.

\bibitem {KMUG4}D. G\'{o}mez-Ullate, N. Kamran, and R. Milson, \emph{Structure
theorems for linear and non-linear differential operators admitting invariant
polynomial subspaces}, Discrete Contin. Dyn. Syst. 18 (2007), no. 1, 85--106.

\bibitem {KMUG5}D. G\'{o}mez-Ullate, N. Kamran, and R. Milson, \emph{A
conjecture on exceptional orthogonal polynomials}, Found. Comput. Math. 13
(2013), no. 4, 615--666.

\bibitem {KMUG6}D. G\'{o}mez-Ullate, N. Kamran, and R. Milson, \emph{On
orthogonal polynomials spanning a non-standard flag,} Contemp. Math., 563
(2012), 51--71.

\bibitem {Gomez-Ullate-Marcellan-Milson}D. G\'{o}mez-Ullate, F. Marcell\'{a}n,
and R. Milson, \textit{Asymptotic and interlacing properties of zeros of
exceptional Jacobi and Laguerre polynomials}, J. Math. Anal. Appl. 399 (2013),
no. 2, 480--495.

\bibitem {Hellwig}G. Hellwig, \emph{Differential operators of mathematical
physics: an introduction, }Addison-Wesley Publishing Co., Reading, Mass., 1964.

\bibitem {Kreyszig}E. Kreyszig, \emph{Introductory Functional Analysis, }John
Wiley and Sons, New York, 1989.

\bibitem {Kwon-Littlejohn}K. H. Kwon and L. L. Littlejohn, \emph{Sobolev
orthogonal polynomials and second-order differential equations}, Rocky
Mountain J. Math\textbf{.}, 28 (2), 1998, 547--594.

\bibitem {Littlejohn-Wellman}L. L. Littlejohn and R. Wellman, \emph{A general
left-definite theory for certain self-adjoint operators with applications to
differential equations}, J. Differential Equations, 181 (2), 2002, 280--339.

\bibitem {Lesky}P. A. Lesky, \emph{\"{U}ber Polynomsysteme, die
Sturm-Liouvilleschen Differenzengleichungen gen\"{u}gen}, Math. Z. 78 (1962) 439--445.

\bibitem {Naimark}M. A. Naimark, \emph{Linear differential operators II,}
Frederick Ungar Publishing Co., New York, 1968.

\bibitem {Odake-Sasaki}S. Odake and R. Sasaki, \emph{Another set of infinitely
many exceptional }$(X_{\ell})$ Laguerre polynomials, Phys. Lett. B 684 (2010), 173--176.

\bibitem {Post-Turbiner}G. Post and A. Turbiner, \emph{Classification of
linear differential operators with an invariant subspace of monomials},
Russian J. Math. Phys. 3 (1995), no. 1, 113--122.

\bibitem {Quesne}C. Quesne, \emph{Exceptional orthogonal polynomials, exactly
solvable potentials and supersymmetry, }J. Phys. A: Math. Theor. 41 (2008)
392001 (6pp).

\bibitem {Routh}E. J. Routh, \emph{On some properties of certain solutions of
a differential equation of the second order}. Proc. London Math. Soc. S1-16
no. 1 (1885), 245--262.

\bibitem {Szego}G. Szeg\"{o}, \emph{Orthogonal polynomials, }American
Mathematical Society Colloquium Publications, Volume \textbf{23}, Providence,
Rhode Island, 1978.
\end{thebibliography}
\end{document}